\documentclass[a4paper,10pt]{article}
\usepackage{theorem}
\usepackage{amsmath}
\usepackage{amsfonts}
\usepackage{amssymb}

\newtheorem{theorem}{Theorem}[section]
\newenvironment{proof}{\textbf{Proof:} }{\text{q.e.d.}}
\newtheorem{assumption}[theorem]{Assumption}
\newtheorem{proposition}[theorem]{Proposition}
{\theorembodyfont{\rmfamily}
\newtheorem{notation}[theorem]{Notation}

\newtheorem{definition}[theorem]{Definition}

\newtheorem{lemma}[theorem]{Lemma}
\newtheorem{corollary}[theorem]{Corollary}
\newtheorem{remark}[theorem]{Remark}

}

\newcommand{\ncmd}{\newcommand}
\ncmd{\be}{\begin{enumerate}}
\ncmd{\ee}{\end{enumerate}}
\ncmd{\bi}{\begin{itemize}}
\ncmd{\ei}{\end{itemize}}
\ncmd{\beq}{\begin{equation}}
\ncmd{\eeq}{\end{equation}}

\ncmd{\lra}{\mbox{$\longrightarrow$}}
\ncmd{\Lra}{\mbox{$\Longrightarrow$}}
\ncmd{\ura}{\underrightarrow}
\ncmd{\ua}{\mbox{$\uparrow$}}
\ncmd{\da}{\mbox{$\downarrow$}}
\ncmd{\hra}{\mbox{$\hookrightarrow$}}
\ncmd{\ra}{\mbox{$\rightarrow$}}
\ncmd{\Ra}{\mbox{$\Rightarrow$}}
\ncmd{\eqa}{\mbox{$\leftrightarrow$}}
\ncmd{\Eqa}{\mbox{$\Leftrightarrow$}}
\ncmd{\leqa}{\mbox{$\longleftrightarrow$}}
\ncmd{\Leqa}{\mbox{$\Longleftrightarrow$}}
\ncmd{\iso}{\mbox{$\tilde{\ra}$}}
\ncmd{\liso}{\mbox{$\tilde{\lra}$}}

\ncmd{\Reinfty}{\mbox{$\mathbb{R}\cup\{\infty\}$}}
\ncmd{\bbN}{\mathbb{N}}
\ncmd{\bbZ}{\mbox{$\mathbb{Z}$}}
\ncmd{\bbR}{\mbox{$\mathbb{R}$}}
\ncmd{\ggT}{\operatorname{ggT}}
\ncmd{\Gcd}[2]{\mathop{\text{gcd}(#1,#2)}}

\ncmd{\mf}[1]{\mathfrak{#1}}
\ncmd{\calli}[1]{\mathcal{#1}}
\ncmd{\Rcal}{\calli{R}}
\ncmd{\Scal}{\calli{S}}
\ncmd{\gfrak}{\mf{g}}
\ncmd{\hfrak}{\mf{h}}

\ncmd{\ti}[1]{\tilde{#1}}
   \ncmd{\tiV}{\ti{V}}\ncmd{\tiD}{\ti{D}}\ncmd{\tisigma}{\ti{\sigma}}
   \ncmd{\tij}{\ti{j}}\ncmd{\tii}{\ti{i}}
   \ncmd{\titfh}{\ti{\mf{h}}}\ncmd{\titfgt}{\ti{\mf{g}}}
   \ncmd{\tiJ}{\ti{J}}\ncmd{\tiJplus}{\ti{J}^+}\ncmd{\tiG}{\ti{G}}
   \ncmd{\tibG}{\mathbf{\ti{G}}} \ncmd{\tibH}{\mathbf{\ti{H}}}
\ncmd{\leftexp}[2]{{\vphantom{#2}}^{#1}{#2}}

\ncmd{\graded}[2]{\mbox{$\mf{gr}_{#1}(#2):=\oplus_i({#1}^i{#2}/{#1}^{i+1}{#2})$}}

\ncmd{\umin}[1]{\underset{#1}{\min}}
\ncmd{\uinf}[1]{\underset{#1}{\inf}}
\ncmd{\usup}[1]{\underset{#1}{\sup}}
\ncmd{\umax}[1]{\underset{#1}{\max}}
\ncmd{\uomin}[2]{\underset{#1}{\overset{#2}{\min}}}
\ncmd{\uoinf}[2]{\underset{#1}{\overset{#2}{\inf}}}
\ncmd{\uosup}[2]{\underset{#1}{\overset{#2}{\sup}}}
\ncmd{\uomax}[2]{\underset{#1}{\overset{#2}{\max}}}

\ncmd{\sequence}[1]{\mbox{$({#1}_n)_\mathbb{N}$}}

\ncmd{\pF}{\mf{p}_F}
\ncmd{\oK}{\mbox{$o_K$}}
\ncmd{\pK}{\mbox{$\mf{p}_K$}}
\ncmd{\pD}{\mbox{$\mf{p}_D$}}
\ncmd{\pFpow}[1]{\mbox{$\mf{p}_F^{#1}$}}
\ncmd{\pDpow}[1]{\mbox{$\mf{p}_D^{#1}$}}
\ncmd{\oDelta}{\mbox{$o_{\Delta}$}}
\ncmd{\pDelta}{\mbox{$\mf{p}_{\Delta}$}}
\ncmd{\vr}[1]{\mbox{$o_{#1}$}}
\ncmd{\vi}[1]{\mbox{$\mf{p}_{#1}$}}
\ncmd{\vipower}[2]{\mbox{$\mf{p}_{#1}^{#2}$}}
\ncmd{\Char}{\operatorname{char}}

\ncmd{\Building}{\mf{B}^1}
\ncmd{\building}{\mf{B}}
\ncmd{\Ical}{\mathcal{I}}
\ncmd{\IAE}{\mbox{$\mathcal{I}^{E^{\times}}$}}
\ncmd{\IB}{\mbox{$\mathcal{I}_E$}}
\ncmd{\Om}[1]{\mbox{$\Omega_{#1}$}}
\ncmd{\OmE}[1]{\mbox{$\Omega_{E,{#1}}$}}
\ncmd{\bary}{\mathop{bary}}
\ncmd{\typ}{\operatorname{typ}}
\ncmd{\Stern}{\operatorname{Stern}}
\ncmd{\Fix}{\operatorname{Fix}}
\ncmd{\diam}{\operatorname{diam}}

\ncmd{\Norm}[2]{\operatorname{Norm}_{#1}(#2)}
\ncmd{\Normone}[2]{\operatorname{Norm}^1_{#1}(#2)}
\ncmd{\Normtwo}[2]{\operatorname{Norm}^2_{#1}(#2)}

\ncmd{\ord}{\oparatorname{ord}}
\ncmd{\Ord}{\operatorname{Ord}}
\ncmd{\lattices}[1]{\operatorname{lattices}({#1})}
\ncmd{\rad}[1]{\operatorname{rad}({#1})}
\ncmd{\Her}[1]{\operatorname{Her}({#1})}
\ncmd{\Latt}[2]{\operatorname{Latt}_{#1}(#2)}
\ncmd{\Gitter}[2]{\operatorname{Latt}(#2,#1)}
\ncmd{\LC}{\operatorname{LC}}
\ncmd{\Lattone}[2]{\operatorname{Latt}^1_{#1}(#2)}
\ncmd{\Latttwo}[2]{\operatorname{Latt}^2_{#1}(#2)}

\ncmd{\Inn}{\operatorname{Inn}}
\ncmd{\Centr}{\operatorname{Z}}
\ncmd{\Z}{\operatorname{Z}}
\ncmd{\lmult}{\operatorname{l}}
\ncmd{\rmult}{\operatorname{r}}
\ncmd{\ind}{\operatorname{ind}}
\ncmd{\Gal}{\mathop{\text{Gal}}}

\ncmd{\LietiG}{\mbox{$\mathop{Lie}(\tilde{G})$}}
\ncmd{\LF}{\operatorname{LF}}
\ncmd{\gyr}{\mf{g}_{x,r}}
\ncmd{\hxr}{\mf{h}_{x,r}}

\ncmd{\Matr}{\operatorname{M}}
\ncmd{\trd}{\operatorname{trd}}
\ncmd{\Nrd}{\operatorname{Nrd}}
\ncmd{\trace}{\operatorname{tr}}
\ncmd{\Gram}{\operatorname{Gram}}
\ncmd{\gr}{\operatorname{\mf{gr}}}
\ncmd{\degree}{\operatorname{deg}}
\ncmd{\Skew}{\operatorname{Skew}}
\ncmd{\Sym}{\operatorname{Sym}}
\ncmd{\diag}{\operatorname{diag}}
\ncmd{\antidiag}{\operatorname{antidiag}}

\ncmd{\EMatr}[1]{\mathds{1}_{#1}}
\ncmd{\Mint}[3]{\mathop{\text{M}_{#1,#2}(#3)}}
\ncmd{\trans}[1]{{#1}^{\ensuremath{\mathsf{T}}}} 

\ncmd{\Lie}{\operatorname{Lie}}
\ncmd{\Abb}{\operatorname{Abb}}
\ncmd{\Top}{\operatorname{Top}}
\ncmd{\Hom}{\operatorname{Hom}}
\ncmd{\End}{\operatorname{End}}
\ncmd{\Aut}{\operatorname{Aut}}
\ncmd{\Iso}{\operatorname{Iso}}

\ncmd{\im}{\operatorname{im}}
\ncmd{\rang}{\operatorname{rang}}

\ncmd{\id}{\operatorname{id}}

\ncmd{\Span}{\operatorname{span}}
\ncmd{\tens}[3]{{#1}\otimes_{#2}{#3}}
\ncmd{\proj}{\operatorname{proj}}

\ncmd{\PO}{\operatorname{PO}}
\ncmd{\SL}{\text{SL}}
\ncmd{\GL}{\text{GL}}
\ncmd{\Sp}{\operatorname{Sp}}
\ncmd{\SO}{\operatorname{SO}}
\ncmd{\Ogp}{\operatorname{O}}
\ncmd{\Rad}{\operatorname{Rad}}
\ncmd{\nr}{\operatorname{nr}}
\ncmd{\GLDV}{\mbox{$\GL_DV$}}
\ncmd{\SLDV}{\mbox{$\SL_DV$}}
\ncmd{\U}{\operatorname{U}}
\ncmd{\SU}{\operatorname{SU}}
\ncmd{\AfSp}[2]{\operatorname{\mathbf{A}}^{#1}_{K}}
\ncmd{\Res}{\operatorname{Res}}
\ncmd{\bV}{\operatorname{\mathbf{V}}}
\ncmd{\bA}{\operatorname{\mathbf{A}}}
\ncmd{\bGm}{\operatorname{\mathbf{G}_{\mathbf{m}}}}
\ncmd{\bG}{\operatorname{\mathbf{G}}}
\ncmd{\bH}{\operatorname{\mathbf{H}}}
\ncmd{\bMatr}{\operatorname{\mathbf{M}}}
\ncmd{\bU}{\operatorname{\mathbf{U}}}
\ncmd{\bSU}{\operatorname{\mathbf{SU}}}
\ncmd{\bSL}{\operatorname{\mathbf{SL}}}
\ncmd{\bO}{\operatorname{\mathbf{O}}}
\ncmd{\bSO}{\operatorname{\mathbf{SO}}}
\ncmd{\bSp}{\operatorname{\mathbf{Sp}}}
\ncmd{\bGL}{\operatorname{\mathbf{GL}}}
\ncmd{\bT}{\operatorname{\mathbf{T}}}
\ncmd{\bAfSp}[1]{\operatorname{\mathbf{A}^{#1}}}
\ncmd{\Weylgroup}{\mbox{$\leftexp{v}{W}$}}
\ncmd{\Osheaf}[1]{\mbox{$\mf{O}_{#1}$}}
\ncmd{\RingedSpace}[1]{\mbox{$(#1,\Osheaf{#1})$}}
\ncmd{\maximalIdeal}[1]{\mbox{$\mf{m}_{#1}$}}

\ncmd{\Invariante}[2]{\mathop{\text{Iv}({#1},{#2})}}
\ncmd{\pairs}{\mathop{\text{pairs}}}

\ncmd{\zzmatrix}[4]{\left( \begin{array}{cc}#1&#2\\#3&#4 \end{array}\right)}
\usepackage{geometry}
\title{The centralizer of a classical group and Bruhat--Tits buildings}
\author{Daniel Skodlerack} 

\begin{document}

\maketitle

\begin{abstract}
Let $G$ be a unitary group defined over a non-Archimedean local field of odd residue characteristic and let $H$ be the centralizer of a semisimple rational  Lie algebra element of $G.$ We prove that the Bruhat--Tits building $\Building(H)$ of $H$ can be affinely and $G$-equivariantly embedded in the Bruhat--Tits building $\Building(G)$ of $G$ so that the Moy--Prasad filtrations are preserved. The latter property forces uniqueness in the following way. Let
$j$ and $j'$ be maps from $\Building(H)$ to $\Building(G)$ which preserve the Moy--Prasad filtrations. We prove that if there is no split torus in the center
of the connected component of $H$ then $j$ and $j'$ are equal, and in general if both maps are affine and satisfy a mild equivariance condition they differ up to a 
translation of $\Building(H).$ 
\end{abstract}
\section{Introduction}
The subject of this article is a functoriality question for maps between Bruhat--Tits buidlings which is connected with the representation theory of classical groups. Embeddings of buildings of reductive groups have previously been studied for example by Landvogt \cite{landvogt:00} and Prasad and Yu \cite{prasadYu:02}. The aim of this article is to show to what extend the property of preserving the Moy--Prasad filtrations determines the 
choice of an embedding uniquely. It completes recent works of Broussous, Lemaire and 
Stevens \cite{broussousLemaire:02}, \cite{broussousStevens:09}, which have applications in representation theory. 

More precisely Bushnell's and Kutzko's strategy in their theory of types for the classification of  irreducible smooth representations of $\GL_n(k)$ in \cite{bushnellKutzko:93} is applied to other classical groups defined over a non-Archimedean local field $k.$ In \cite{secherreI:04}, \cite{secherreII:05}, \cite{secherreIII:05}, \cite{secherreStevensIV:08}, \cite{broussousSecherreStevens:10} and \cite{secherreStevensVI:10}  S\'echerre together with Broussous and Stevens gave the classification for $\GL_n(D),$ where $D$ is a central finite division algebra over $k.$ Further in \cite{stevens:05} Stevens constructed types for unitary groups. He applied a result of his paper with Broussous \cite{broussousStevens:09}, i.e. he used an embedding of Bruhat--Tits buildings for a certain subgroup of a unitary group to apply an induction. The important property of this map is the compatibility with the Lie algebra filtrations (\textbf{CLF}) which by \cite{lemaire:09} correspond to the Moy--Prasad filtrations \cite{moyPrasad:94}. 
In \cite{broussousStevens:09} the quaternion algebra case is missing and the authors proposed a uniqueness and generalization conjecture to the reader. 

Let $k$ be a non-Archimedean local field with valuation $\nu$ and of residue characteristic $\neq 2,$ and let $D$ be a division algebra central and finite-dimensional over $k$ equipped with an involution $\rho,$ whose set of fixed points in $k$ is denoted by $k_0.$ To the classical group $\bG:=\bU(h),$ i.e. the unitary group of an $\epsilon$-hermitian form $h$ on a finite-dimensional right $D$-vector space $V,$ is attached the Bruhat--Tits building $\Building(\bG,k_0)$ which can be described in terms of lattice functions. To every point of $\Building(\bG,k_0)$ there is attached a Lie algebra filtration $\mf{g}_x$ which is exactly the Moy--Prasad filtration. More precisely if $x$ is a point of $\Building(\bG,k_0)$ seen as a lattice function it can be interpreted 
as a point of $\building(\bGL_D(V),k)$ which has a Lie algebra filtration $\ti{\mf{g}}_x$ in $\End_D(V).$ If one identifies $\Lie(\bG)(k_0)$ with the set of skew-symmetric elements of $\End_D(V)$ with respect to the adjoint involution of $h,$ the filtration $\mf{g}_x$ of $x$ coincides with 
\[t\mapsto \ti{\mf{g}}_x(t)\cap \Lie(\bG)(k_0).\]
Now we take an element $\beta\in\Lie(\bG)(k_0)$ whose $k$-algebra $k[\beta]$ is a product of field extensions $E_i$ of $k.$ In this introduction let us assume $k[\beta]$ to be separable over $k.$ 

Its centralizer in $\bG$ is an algebraic group $\bH$ defined over $k_0$ which is a product of restrictions of scalars to $k_0$ of classical groups $\bH_i$ which are either general linear groups or unitary groups. In the manner of \cite{broussousStevens:09} we prove the existence of an injective affine $\bH(k_0)$-equivariant and toral map $j_{\beta}$ from $\Building(\bH,k_0)$ to $\Building(\bG,k_0)$ using lattice function models.
In addition $j_{\beta}$ has the CLF-property,  i.e. 
\[\Lie(\bH)(k_0)\cap \mf{g}_{j_{\beta}(y)}\]
is the Lie algebra flitration $\mf{h}_x$ of $x.$ This article solves the following problem:
To what extent does the CLF-property determine $j_{\beta}?$

In order to give an answer to this question we consider two cases.
\begin{enumerate}
 \item[(a)] If only unitary groups appear among the $\bH_i,$ none of which is $k_0$-isomorphic to the isotropic orthogonal group of 
rank one, then $j_{\beta}$ is uniquely determined by CLF.
 \item[(b)] If there are no restrictions on the $\bH_i$ then $j_{\beta}$ is unique up to a translation of $\Building(\bH,k_0)$
in being affine, $\Centr(\bH^0(k_0))$-equivariant and having the CLF-property.
\end{enumerate}

For the proofs we use the decomposition of $k[\beta]$ as a product of fields to restrict to the cases where $k[\beta]$ is a field or a product of two fields interchanged by the adjoint involution of $h.$ Let us call these cases atomic cases. The map $j_{\beta}$ is now induced by $j_E^{-1}$ constructed in \cite{broussousLemaire:02}.  
In (a) $k[\beta]$ is a field in the atomic case. If $\beta$ is non-zero we follow the strategy of \cite{broussousStevens:09}. The statement follows essentially from a uniqueness statement for $j_E^{-1}$ given in \cite[10.3]{broussousStevens:09}.
If $\beta$ is zero we prove that for all unitary groups except for the  isotropic orthogonal group of rank one the Moy--Prasad filtration determines the point completely.
For (b) to restrict to the atomic case we need further a rigidity proposition for Euclidean buildings, stated in \ref{propRigidityForAffineBuildings}. We use 
(a) to finish the proof of (b).

The whole strategy and the proofs do not require $\beta$ to be separable. In that case we define the building of 
$\Centr_{\bG(k_0)}(\beta)$ in view of \cite{broussousStevens:09} using lattice functions and we work mainly with the rational points instead of the algebraic groups. 

The article is structured in the following way. After preliminary notation in \S2 and the model of the Bruhat--Tits building of $\bG$ over $k_0$ in terms of lattice functions is explained in \S3.  In \S4 we introduce the Moy--Prasad filtration for our purposes. We give the building of the centralizer in \S5 and introduce the notion of CLF in \S6 followed by the existence theorem in \S7. The uniqueness theorems in \S9, where no $\GL_m$ is involved, and \S11, for the general case, are prepared in the preceding sections. Lastly in \S12 we show that the constructed map respects apartments.

I thank very much P. Broussous, Prof. S. Stevens and Prof. E.-W. Zink for useful hints and fruitful communications. I want to express my gratitude to the German Research Foundation, who supported this work within the framework of BMS and SFB 878.

\section{Notation}\label{secNotation}
We are given a division algebra $D$ of finite index $d$ and central over a non-Archimedean local field $k$ of odd residue characteristic. The valuation on $k$ and its unique extension to $D$ are denoted by $\nu.$ We assume that the image of $\nu|_k$ is $\bbZ.$ Further let $\rho$ be an involution on $D,$ i.e. a skew field isomorphism from $D$ to $D^{op}$ of order one or two, in particular $\rho$ is an isometry. The fixed field of $\rho$ in $k$ is denoted by $k_0.$ 

\begin{remark}\cite[10.2.2]{scharlau:85}\label{remd12}
 The existence of $\rho$ implies $d\leq 2.$ If $k\neq k_0,$ then $d=1.$ 
\end{remark}

We fix an element 
$\epsilon\in\{1,-1\}$ and a non-degenerate $\epsilon$-hermitian form $h$ on an $m$-dimensional right $D$-vector space $V,$ i.e. a $\bbZ$-bi-linear map $h$ on $V\times V$ with values in $D$
 such that 
 \[h(v_1\lambda_1,v_2\lambda_2)=\epsilon\rho(\lambda_1)\rho(h(v_2,v_1))\lambda_2\]
 for all $\lambda_1,\lambda_2\in D$ and $v_1,v_2\in V.$ Further $\rho|_k$
 extends to the adjoint involution $\sigma$ of $h$ on 
 $\End_D(V).$
%
For a skew field $D'$ with discrete valuation the symbols $o_{D'}$ and $\mf{p}_{D'}$ denote the valuation ring and its maximal ideal respectively. We fix an algebraic closure $\bar{k}$
of $k.$ 

By a Bruhat--Tits building we alway understand the extended one \cite[4.2.16.]{bruhatTitsII:84}.
The set
\[G:=\U(\sigma):=\{f\in\End_D(V)|\ \sigma(f)f=\id_V\}\]
is the set of $k_0$-rational points of an algebraic group $\bU(\sigma)$ defined over $k_0.$
We also write $\bU(h)$ for $\bU(\sigma).$ We denote $\bU(h)$ by $\bG$ and $\Lie(\bG)(k_0)$ by $\mf{g}$ and identify the latter with the set $\Skew(\End_D(V),\sigma)$ of skew-symmetric elements of $\End_D(V)$ with respect to $\sigma.$
The Bruhat--Tits building of $\Building(\bG,k_0)$ is also denoted by $\Building(G).$
We repeat the strategy in \cite{broussousStevens:09} to describe the building as a set of self-dual lattice functions based on the description using norms, see \cite{bruhatTitsIV:87}.
\section{Norms and lattice functions}

The description of the Bruhat-Tits building in terms of norms and lattice functions needs some basic properties which are collected in this section. Proofs
and more details for norms can mainly be found in \cite{bruhatTitsIII:84}, \cite{bruhatTitsIV:87}. For lattice functions we refer to \cite{broussousStevens:09} and \cite{broussousLemaire:02}.

\begin{definition}{\cite[1.1]{bruhatTitsIII:84}}
A $D$-norm on $V$ is a map $\alpha:\ V\ra \bbR\cup\{\infty\}$ such that
\bi
\item $\alpha(v)=\infty$ if and only if $v=0,$
\item $\alpha(v\lambda)=\nu(\lambda)+\alpha(v),$
\item $\alpha(v+w)\geq\inf\{\alpha(v),\alpha(w)\}$,
\ei
for all $v,w\in V$ and $\lambda\in D.$ Fix a $D$-norm $\alpha$ on $V.$
The dual $\alpha^{\#}$ of $\alpha$ with respect to $h$ is defined to be the $D$-norm
\[v\in V\mapsto \inf\{\nu(h(v,w))-\alpha(w)|\ w\in V,\ w\neq 0\}.\] One calls $\alpha$ {\it self-dual} with respect to 
$h$ if it is equal to his dual. The set of $D$-norms and self-dual $D$-norms of $V$ are denoted by $\Normone{D}{V}$ and 
$\Normone{h}{V}$ respectively. 
Given a further right $D$-vector space $V'$ and a norm $\alpha'\in\Normone{D}{V'}$ the direct sum of $\alpha$ and $\alpha'$ is defined by 
\[(\alpha\oplus\alpha')(v+v'):=\inf\{\alpha(v),\alpha'(v')\},\ v\in V,\ v'\in V'.\]
\end{definition}

We consider a set $\Rcal$ of $D$-subspaces of $V$ whose direct sum is $V.$
We call $\Rcal$ a frame if all elements of $\Rcal$ are one dimensional.
An element $\alpha\in\Normone{D}{V}$ is {\it split} by $\Rcal$ if $\alpha$ is the direct sum of the 
$\alpha|_W,\ W\in\Rcal.$
If in addition $\Rcal$ is a frame and $(w_i)_i$ is a $D$-basis of $V$ consisting of elements of 
$\bigcup_{W\in\Rcal}W$ one calls  $(w_i)_i$ a splitting basis of $\alpha.$

\begin{remark}\cite[1.26]{bruhatTitsIII:84}
Every pair of $D$-norms on $V$ has a common splitting basis.
\end{remark}

Given two norms $\alpha,\ \gamma\in\Normone{D}{V}$ with common splitting basis $(v_i)$ and a real number
$\lambda\in [0,1]$ the convex combination of $\alpha$ with $\gamma$ with $\lambda$ is defined to be the 
$D$-norm on $V$ split by $(v_i)$ whose value at $v_i$ is 
\[\lambda\alpha(v_i)+(1-\lambda)\gamma(v_i).\]
This definition does not depend on the choice of the splitting basis and we get an affine structure on $\Normone{D}{V}.$
The $\Aut_D(V)$-action on $\Normone{D}{V},$ more precisely
\[(g.\alpha) (v):=\alpha(g^{-1}v),\]
restricts to a $\U(\sigma)$-action on $\Normone{h}{V}.$

The family of balls of a $D$-norm $\alpha$ leads to the idea of a lattice function:
\[\Lambda_{\alpha}(t):=\{v\in V|\ \alpha(v)\geq t\},\ t\in\bbR.\]
Before we give the definition we want to remark that by an  {\it $o_D$-lattice} in $V$ we mean a finitely generated $o_D$-submodule of $V$ which contains a $D$-basis of $V,$ i.e. we omit the word ''full''.

\begin{definition}{\cite[I.2.1]{broussousLemaire:02}}
A family $(\Lambda(r))_{r\in\bbR}$ of $o_D$-lattices in $V$ is said to be an $o_D$-lattice function in $V$ if 
\be
\item $\Lambda(r)\supseteq \Lambda(s),$
\item $\Lambda(r)=\bigcap_{t<r}\Lambda(t)$ and 
\item $\Lambda(r)\mathfrak{p}_D=\Lambda(r+\frac1d)$,
\ee  
for all $r,s\in\bbR$ with $r<s.$ The set of $o_D$-lattice functions in $V$ is denoted by $\Lattone{o_D}{V}.$
For the right limit of $\Lambda$ at $s$ we write $\Lambda(s+),$ i.e. 
\[\Lambda(s+):=\bigcup_{t>s}\Lambda(s).\] 
\end{definition}

For $t\in\bbR$ we define $\lceil t\rceil$ to be smallest integer not smaller than $t.$

\begin{proposition}{\cite[I.2.4]{broussousLemaire:02}}\label{propLattNorm1}
The map 
\[\alpha\mapsto\Lambda_{\alpha}\]
is a bijection from $\Normone{D}{V}$ to $\Lattone{o_D}{V}.$ Its inverse is given by 
\[\alpha_{\Lambda}(v):=\sup\{t\in\bbR |\ v\in\Lambda(t)\}.\] 
\end{proposition}

All notions for norms carry over to lattice functions in the following way.
The {\it dual}\/ of a lattice function $\Lambda$ is the lattice function which corresponds to $\alpha^{\#}_{\Lambda}.$ A lattice function is called {\it split}\/ by a given basis if the correspoding norm is split by this basis. 
The $\Aut_D(V)$-action on $\Normone{D}{V}$ defines an $\Aut_D(V)$-action on 
$\Lattone{o_D}{V}$ via push forward.

\begin{proposition}{\cite[3.3]{broussousStevens:09},\cite[I.2.4]{broussousLemaire:02}}\label{propLattNorm2}
Let $\Lambda$ be an $o_D$-lattice function in $V.$ Let $g$ be an element of $\Aut_D(V)$ and let $\Rcal$ be a set of $D$-subspaces of $V$ whose direct sum is $V.$
\be
\item The function $\Lambda$ is split by $\Rcal$ if and only if, for all real numbers $t$, the lattice $\Lambda(t)$ is the direct sum of the $o_D$-modules $W\cap\Lambda(t),$ $W\in\Rcal.$
\item If $(v_i)_i$ is a splitting basis of $\Lambda,$ then there are real numbers $a_i,\ 1\leq i\leq m,$ such that, for all $t$, 
we have
\[\Lambda(t)=\bigoplus_iv_i\mf{p}_D^{\lceil(t-a_i)d\rceil}.\]
We call $(a_i)_i$ the {\it coordinate tuple} of $\Lambda$ with respect to $(v_i).$ The map which assigns the coordinate tuple to a lattice function split by $(v_i)$ is an affine bijection onto $\bbR^m.$
\item We have $(g.\Lambda)(t)=g(\Lambda(t)).$
\item The dual $\Lambda^{\#}$ of $\Lambda$ with respect to $h$ is the $o_D$-lattice function whose value at $t\in\bbR$ is
\[\{v\in V|\ h(v,\Lambda((-t)+))\subseteq \mf{p}_D\}.\] 
\ee
\end{proposition}

\begin{proof}
All statements except for 4. can be found in part I of \cite{broussousLemaire:02}. Point 4. has been proven in 
\cite[3.3]{broussousStevens:09} for the case $D=k,$ but the proof is valid for the general case if one replaces $o_F$ by $o_D.$ 
\end{proof}

The set of self-dual $o_D$-lattice functions is denoted by $\Lattone{h}{V}.$ We recall that we have fixed an $\epsilon$-hermitian form $h$ on $V.$ 

\begin{theorem}{\cite[2.12]{bruhatTitsIV:87}}
There is a unique affine and $G$-equivariant bijection from 
$\Building(G)$ to $\Normone{h}{V}.$
\end{theorem}

Propositions \ref{propLattNorm1} and \ref{propLattNorm2} imply the existence of a unique affine and $G$-equivariant bijection from $\Building(G)$ to $\Lattone{h}{V}.$ It defines on $\Lattone{h}{V}$ a system of apartments which correspond to the
Witt-decompositions of $V.$
\begin{definition}
A set of one dimensional $h$-isotropic $D$-subspaces $W_l$ of $V,\ l\in L,$ is said to be a {\it Witt decomposition of $V$} if: 
\be
\item for every $l\in L$ there is exactly one index $l'\in L$ such that $h(W_l,W_{l'})$ is non-zero,
\item the sum of the $W_l$ is direct, and
\item the orthogonal complement $W$ of the sum of the $W_l$ is anisotropic. 
\ee
A lattice function is said to be {\it split} by a Witt decomposition $\{W_l|\ l\in L\}$ if it is split 
by $\{W_l|\ l\in L\}\cup\{W\}.$ 
A $D$-basis of $V$ is {\it adapted to} $\{W_l|\ l\in L\}$ if all basis elements lie in the union of $W$ and all $W_l.$ Further we assume $0\not\in L$ and denote
$W$ by $W_0.$
\end{definition}

Witt decompositions always exists, and even more, for every element of $\Lattone{h}{V}$ there is a splitting Witt decomposition. A proof for the latter in terms of norms can be found in \cite[2.13]{bruhatTitsIV:87}.
 
\begin{remark}{\cite[2.9]{bruhatTitsIV:87}}
\be
 \item The system of apartments of $\Lattone{h}{V}$ is the system of sets
\[\Lattone{h,\Scal}{V}:=\{\Lambda\in\Lattone{h}{V}|\ \Lambda\text{ is split by }\Scal\},\]
where $\Scal$ runs over the set of Witt decompositions of $V$ with respect to $h.$
\item If a self dual $D$-norm $\alpha$ split by a Witt decomposition $\{W_l|\ l\in L\}$ satisfies
\[\alpha(w_0)=\frac12\nu(h(w_0,w_0)),\]
for all $w_0\in W_0$ (see \cite[2.9]{bruhatTitsIV:87}).
\ee 
\end{remark}

\begin{proposition}\label{propNormSplitByOGBasis}
If $\alpha$ is a self dual $D$-norm on $V$ and $\{W_l|\ l\in L\}$ is a Witt decomposition splitting $\alpha,$
then every orthogonal $D$-basis of $W_0$ splits $\alpha|_{W_0}$.
\end{proposition}

\begin{proof}
Let $w'$ and $w''$ be elements of $W_0$ which are orthogonal to each other, then we have
\[\alpha\left(w'\right)=\alpha\left(\tfrac12\left(w'+w''\right)+\tfrac12\left(w'-w''\right)\right)\geq\inf\{\alpha\left(\tfrac12\left(w'+w''\right)\right),\alpha\left(\tfrac12\left(w'-w''\right)\right)\}=
\alpha\left(w'+w''\right),\]
and the last equality follows from
\[h\left(\tfrac12\left(w'+w''\right),\tfrac12\left(w'+w''\right)\right)=h\left(\tfrac12\left(w'-w''\right),\tfrac12\left(w'-w''\right)\right)\]
and $\nu(2)=0.$
\end{proof}

\section{The Moy--Prasad filtrations}
An $o_D$-lattice function $\Lambda$ gives rise to an $o_k$-lattice function in $A:=\End_D(V)$
\[\tilde{g}_{\Lambda}(t):=\{a\in A|\ a(\Lambda(s))\subseteq\Lambda(s+t),\ \forall s\in\bbR\}\]
called the {\it square lattice function} of $\Lambda$ in $A.$ It defines a Lie algebra filtration
of $\Lambda$ in $\mf{g}$ by intersection
\[\mf{g}_{\Lambda}(t):=\tilde{g}_{\Lambda}(t)\cap \mf{g}.\]

In \cite{moyPrasad:94} Moy--Prasad attached to every point $x$ of $\Building(G',k')$ a filtration 
$(\mf{a}_{x,t})_{t\in\mathbb{R}}$ of $\Lie(G')(k')$ for a reductive group $G'$ defined over a non-Archimedean local field $k'.$ 
The next theorem states that in our special situation of unitary groups we do not need the quite involved description.

\begin{theorem}{\cite{lemaire:09}}
The Lie algebra filtration of a point of $\Building(G)$ is exactly its Moy--Prasad filtration.
\end{theorem}

In this article the square lattice functions are the key for the rigidity of Lie algebra filtrations because of the 
following proposition.

\begin{proposition}{\cite[3.5]{broussousStevens:09}}
Let $\Lambda$ be an $o_D$-lattice function in $V.$
\be
\item The $o_k$-lattice function $(\sigma(\tilde{g}_{\Lambda}(t)))_{t\in\bbR}$ is the square lattice of $\Lambda^{\#}$ and we denote it by $\tilde{g}_{\Lambda}^{\sigma}.$ Square lattice functions fixed under $\sigma$ are said to be {\it 
self-dual.}
\item The map 
\[\Lambda\in \Lattone{h}{V}\mapsto\tilde{g}_{\Lambda}\]
is injective and onto to the set $\Latttwo{\sigma}{A}$ of self-dual square  lattice functions.
\ee
\end{proposition}

\begin{proof}
The proof in \cite[3.5]{broussousStevens:09} is valid if one replaces $F$ by $D.$
\end{proof}

\section{The centralizer}\label{secCentralizer}
In the following the symbol $\Centr_{*}(?)$ denotes the centralizer of ? in *.
Let us fix an element $\beta\in\Skew(\End_D(V),\sigma),$ such that the $k$-algebra $k[\beta]$ is a product of fields $E_i,$ $i\in I,$ with identity element $1_i.$ We call $\beta$ {\it separable} if $E_i|1_ik$ is separable for all $i.$ The centralizers $\Centr_{G}(\beta)$ and $\Centr_{\bG}(\beta)$  are denoted by $H$ and $\bH.$ 
If $\beta$ is separable then the algebraic group $\bH$ is reductive and defined over $k_0$ and its set of $k_0$-rational points is $H$ (see appendix \ref{appendixCentralizer}).

\begin{notation}\label{notActionOfSigmaOnI}
There is an action of $\sigma$ on $I$ via 
$\sigma(1_i)=:1_{\sigma(i)}.$ We denote by $I_0$ the fixed point set of the action of $\sigma$ on $I,$ and we divide the set $I\setminus I_0$ into two disjoint parts, i.e.
we choose a positive part $I_+$ and a negative part $I_-,$ such that 
\[\sigma(I_+)=I_-.\]
We define $-i:=\sigma(i)$ and put $V_i:=1_iV.$ For $i\in I_0$ we denote by $(E_i)_0$ the set of fixed points of $\sigma$ in $E_i,$ and for $i\in I_+$ we put 
$(E_i)_0:=E_i.$
\end{notation}

The group $H$ is the product of sets of rational points of classical groups over the $(E_i)_0.$ More precisely, for $i\in I,$ the $E_i$-algebra $\End_{\tens{E_i}{k}{D}}(V_i)$ is $E_i$-algebra isomorphic to 
$\End_{D'_i}(V'_i),$ for some finite-dimensional vector space $V'_i$ over some skew field $D'_i$
 central and of finite index over $E_i.$ We define
\[\bH_i:=\left\{\begin{array}{cc}
                \bU(\sigma|_{\End_{D'_i}(V'_i)}),& i\in I_0\\
                \bGL_{D'_i}(V'_i),& i\in I_+\\
               \end{array}\right.
\]
and $H_i:=\bH_i((E_i)_0).$
There is a canonical group isomorphism from $H$ to $\prod_{i\in I_+\cup I_0}H_i$ which motivates the following definition for the Bruhat--Tits building of $H:$
\[\Building(H):=\prod_{i\in I_+\cup I_0}\Building(\bH_i,(E_i)_0).\] 

\begin{remark}
If $\beta$ is separable the above isomorphism extends to a $k_0$-isomorphism 
\begin{eqnarray}
\bH \cong \prod_{i\in I_0\cup I_+}\Res_{(E_i)_0|k_0}(\bH_i)(k_0).
\end{eqnarray}
Thus there is a isomorphism of affine buildings from 
$\Building(\bH,k_0)$ to $\Building(H).$ 
\end{remark}

The Lie algebra of $H$ is isomorphic to 
\[\bigoplus_{i\in I_0}\Skew(\End_{\tens{E_i}{k}{D}}(V_i),\sigma)\oplus\bigoplus_{i\in I_+}\End_{\tens{E_i}{k}{D}}(V_i)\]
and we denote the $i$th factor by $\mf{g}_i.$ 
We define the Lie algebra filtration of $x=(x_i)_{i\in I_0\cup I_+}$ to be the direct sum of the Lie algebra filtrations of the $x_i,$ i.e.
\[\mf{h}_x(t):= \bigoplus_{i\in I_0\cup I_+}(\mf{g}_i)_{x_i}(t).\]
where we take the square lattice function for $i\in I_+$ and identify $\End_{\tens{E_i}{k}{D}}(V_i)$ with $\End_{D'_i}(V_i).$
\section{Compatibility with the Lie algebra filtrations}
\begin{definition}
\be
\item A {\it set with Lie algebra filtrations (LF-set)} is a pair $(X,L)$ such that $X$ is a set, $L$ is a Lie algebra, and for each $x\in X$, there is Lie algebra filtration $((L_x)(t))_{t\in\bbR},$ i.e. a decreasing sequence of subsets of $L.$
\item Given two LF-sets $(X,L)$ and $(X',L')$ and a map $\phi:L\ra L'$ 
\be
\item a map $f:X\ra X'$ is {\it compatible with the Lie algebra filtrations (CLF)} if 
\[\phi(L_x(t))=\im(\phi)\cap L'_{f(x)}(t)\]
holds, for all $x\in X$ and all real numbers $t$.
\item  $g:X'\ra X$ is {\it compatible with the Lie algebra filtrations (CLF)} if 
\[\phi(L_{g(x')}(t))=\im(\phi)\cap L'_{x'}(t)\]
holds, for all $x'\in X'$ and all real numbers $t$.
\ee
\ee
\end{definition}

In this article we consider buildings together with the set of rational points of a Lie algebra of an algebraic group.
Mostly $\phi$ is a canonical inclusion.
An example of a CLF-map is given in \cite{broussousLemaire:02}. The theorems in this section are valid without the assumption
on the residue characteristic and the existence of an involution on $D.$ Firstly, we introduce two important LF-sets after the following theorem.

\begin{definition}
\be
\item Two $o_D$-lattice functions $\Lambda$ and $\Lambda'$ on $V$ are {\it equivalent} if there is a real number $t$ such that 
\[\Lambda(s)=\Lambda'(s-t)=:(\Lambda'+t)(s),\]
for all $s\in\bbR.$ The set of equivalence classes of elements of $\Lattone{o_D}{V}$
is denoted by $\Latt{o_D}{V}.$ A map $f$ on $\Lattone{o_D}{V}$ of the form 
\[f(\Lambda)=\Lambda+t\]
is called a {\it translation}. For a map $g,$ such that $f\circ g$ exists, we define
\[g+t:=f\circ g.\]

\item An action of $\Aut_D(V)$ on $\Latt{o_D}{V}$ is given by 
\[a.[\Lambda]:=[a.\Lambda],\ \text{where }(a.\Lambda)(t):=a(\Lambda(t)).\] 
\item The affine structure on $\Lattone{o_D}{V}$ induces an affine structure
on $\Latt{o_D}{V}.$
\ee
\end{definition}

Let us consider the non-extended building $\building(\Aut_{D}(V))$ and the 
extended building $\Building(\Aut_{D}(V)).$

\begin{theorem}{\cite[I.2.4]{broussousLemaire:02},\cite[2.13]{bruhatTitsIII:84},\cite[2.11 (iii)]{bruhatTitsIII:84}}
The extended building of $\Aut_{D}(V)$ is in affine and $\Aut_D(V)$-equivariant bijection with the set $\Lattone{o_D}{V}.$
For two such $\Aut_D(V)$-equivariant affine bijections $f$ and $g$, there is a real number $t$, such that 
\[f=g+t,\] 
and the map $f$ induces via 
\[[x]\mapsto [f(x)]\]
a unique affine and $\Aut_{D}(V)$-equivariant
bijection from $\building(\Aut_{D}(V))$ to $\Latt{o_D}{V}.$
\end{theorem}

\begin{remark}
For $x\in\Building(\Aut_{D}(V))$, we attach to $x$ and to $[x]$ the square lattice function of a correspoding $o_D$-lattice function. By the above theorem 
both LF-sets $(\Building(\Aut_{D}(V)),\End_D(V))$ and $(\building(\Aut_{D}(V)),\End_D(V))$ are well-defined, i.e. do not depend on the 
choices made. In $(\building(\Aut_{D}(V)),\End_D(V)),$ a point is uniquely determined by its Lie algebra filtration (see \cite[I.4.5]{broussousLemaire:02}), whereas in $(\Building(\Aut_{D}(V)),\End_D(V))$ two points with the same Lie algebra filtration are translates of each other.
\end{remark}

\begin{theorem}{\cite[II.1.1]{broussousLemaire:02}}\label{thmBL02II.1.1.}
Let $E|k$ be a field extension in $\End_D(V),$ $\ti{G}:=\Aut_{D}(V)$ and $\ti{H}:=\Centr_{\ti{G}}(E).$ 
We consider the non-extended buildings $\building(\ti{G})$ and $\building(\ti{H})$ as LF-sets as constructed in the above remark. Then there is a unique CLF-map $j_E$ from
$\building(\ti{G})^{E^{\times}}$ to $\building(\ti{H}).$ The map has the following properties.
\be
\item It is bijective.
\item It is $\ti{H}$-equivariant.
\item It is affine.
\ee
The map $j^E:=(j_E)^{-1}$ is the unique map satisfying 2. and 3..
\end{theorem}

In the above setting the CLF-property of $j^E$ implies uniqueness.

\begin{theorem}\cite[10.3]{broussousStevens:09}\label{thmBS10.3}
Under the assumptions of Theorem \ref{thmBL02II.1.1.}, suppose that the points $y\in\building(\ti{G})$ and $x\in\building(\ti{H})$ satisfy
\[\ti{\mf{g}}_{y}\cap \End_{\tens{E}{k}{D}}(V)\supseteq\ti{\mf{h}}_x.\]
Then $j^E(x)=y.$
\end{theorem}

\begin{proof}
In \cite[10.3]{broussousStevens:09} this theorem was proven for the case where
$D=k$ and $E$ is generated by one element. The proof did not use the second assumption, and it carries over to $D\neq k$ without changes. 
\end{proof}

The map $j^E$ is induced from a map between the extended buildings. We fix an  $E$-algebra isomorphism from $\End_{\tens{E}{k}{D}}(V)$ to $\End_{D'}(V')$ where
$D'$ is a central skew field over $E$ and $V'$ is a right $D'$-vector space.

\begin{theorem}{\cite[II.3.1,II.4]{broussousLemaire:02},\cite[2.11 (iii)]{bruhatTitsIII:84}}\label{thmBL02II.3.1II.4}
There is a bijective affine $\ti{H}$-equivariant CLF-map 
\[\ti{j}^E:\ \Lattone{o_{D'}}{V'}\ra\Lattone{o_D}{V}^{E^{\times}}\]
such that 
\[[\ti{j}^E(\Lambda)]=j^E([\Lambda])\]
for all $\Lambda\in\Lattone{o_{D'}}{V'}.$
The image is the set of $o_D$-lattice functions which are in addition $o_E$-lattice functions. For every other bijective affine $\ti{H}$-equivariant map $j$ from $\Lattone{o_{D'}}{V'}$ to $\Lattone{o_D}{V}^{E^{\times}}$ the composition
$j^{-1}\circ\ti{j}^E$ is a translation of  $\Lattone{o_{D'}}{V'}.$
\end{theorem}

\begin{proof}
The existence of $\ti{j}^E$ is stated in Lemma \cite[II.3.1]{broussousLemaire:02}. 
The affineness is proven in \cite[II.4]{broussousLemaire:02}
for $j_E,$ but the proof actually shows that $\ti{j}^E$ is affine. The $\ti{H}$-equivariance 
follows from the formula given in \cite[II.3.1]{broussousLemaire:02}. 
The fact that the image of $j^E$ is the set of $E^{\times}$-fixed points implies that $\im(\ti{j}^E)$ only consists of $o_E$-lattice functions and contains a representative for every element of $\im(j^E).$ 
Being $E^{\times}$-equivariant and convex, the image of $\ti{j}^E$ must contain every $o_E$-lattice function of $\Lattone{o_D}{V}.$ The last assertion follows directly from
\cite[2.11 (iii)]{bruhatTitsIII:84}.
\end{proof}
\section{CLF-map from $\Building(H)$ to $\Building(G)$}\label{secCLFExistence}

We are now returning to the situation of section \ref{secCentralizer}.
Before we state the first theorem we give a definition in analogy with the set of fixed points in the building.

\begin{definition}
An \textit{$o_E$-$o_D$-lattice function} of $V$ is an $o_D$-lattice function in $V$ 
which splits under $(V_i)$ such that, for every $i$, the function
\[t\mapsto \Lambda(t)\cap V_i\]
is an $o_{E_i}$-lattice function in $V_i.$ We denote the set of $o_E$-$o_D$-lattice functions
by $\Lattone{o_E,o_D}{V}.$
\end{definition}

The next theorem has been proven for $D=k$ in \cite{broussousStevens:09}.

\begin{theorem}\label{thmExistence}
There is an injective, affine and $H$-equivariant CLF-map
\[j:\ \Building(H)\ra\Building(G)\]
whose image in terms of lattice functions is the set of self-dual $o_E$-$o_D$-lattice functions in $V.$
\end{theorem}

\begin{proof}
The product decomposition $H=\prod_{i\in I_+\cup I_0}\Centr_{\U(h_i|_{(V_i+V_{-i})^2})}(\beta_i+\beta_{-i})$ leads us to three steps:
\be
\item the case where $E$ is a field;
\item the case where $I_+$ has cardinality one;
\item the general case. 
\ee

Step 1: We use the map $j^E$ from Theorem \ref{thmBL02II.1.1.} and the following diagram.
\[\begin{array}{ccc}
\Latt{o_{D'}}{V'}& \stackrel{j^E}{\ra} &\Latt{o_D}{V}\\
\uparrow & & \uparrow\\
\Building(H)& & \Building(G)\\
\end{array}\]
We have to prove that, for $x\in\Building(H)$, the square lattice function of $j^E(x)$ is self-dual. The latter follows from Theorem \ref{thmBS10.3}
because the square lattice function of $x$ is self-dual and $j^E$ is a CLF-map. The assertion about the image 
of $j^E|_{\Building(H)}$ follows from \ref{thmBL02II.1.1.} and the CLF-property.

Step 2: To prove the lemma we denote the unique element in $I_+$ by $i$ and we define
\[\Lambda^{\#_{-i}}(t):=\{v\in V_{-i}|\ h(v,\Lambda((-t)+))\subseteq \mf{p}_D\}\]
for $\Lambda\in\Lattone{o_D}{V_i}.$ 
The map
\beq\label{eqEmbedGroup}
g\in\Aut_D(V_i)\mapsto (g,0)+\sigma((g^{-1},0))\in\Aut_D(V_i)\oplus\Aut_D(V_{-i})
\eeq defines
a $k$-embedding from $\bGL_{D}(V_i)$ to $\bG$ mapping $\Centr_{\GL_{D}(V_i)}(\beta_i)$ onto 
$H.$ An injective, affine and $\Aut_D(V_i)$-equivariant CLF-map from $\Building(\bGL_{D}(V_i),k)$ to 
$\Building(G)$ in terms of lattice functions is given by 
\beq\label{eqEmbedBuild}
\Lambda\in\Lattone{o_D}{V_i}\mapsto (\Lambda\oplus\Lambda^{\#_{-i}})\in\Lattone{h}{V}.
\eeq
Now we apply Theorem \ref{thmBL02II.3.1II.4} to finish step 2.

Step 3: We take the direct sum of the maps constructed in the steps before. \end{proof}
\section{Factorization}
In order to analyze the set of CLF-maps from $\Building(H)$ to $\Building(G)$ this section reduces the problem to the case
Where $I_0\cup I_+$ has cardinality one. 
\begin{lemma}\label{lem1betai0}
There is at most one index $i\in I$ such that $\beta_i=0$
and if such an index exists, it has to be in $I_0.$
\end{lemma} 

\begin{proof}
Assume that $\beta_i$ and $\beta_j$ are zero for two different indices $i$ and $j.$ Taking a polynomial $P$ with coefficients in $k$ 
such that $1_i=P(\beta)$ we obtain firstly
\[1_i=1_i1_i=1_iP(\beta)=1_iP(0)\]
and secondly
\[0=1_i1_j=P(\beta)1_j=P(0)1_j\] which is a contradiction. 
The second assertion follows from $-\beta_{-i}=\sigma(\beta_i)$ and the uniqueness.
\end{proof}

Let us recall that the embedding of $\mf{h}$ into $\mf{g}$ is realized by mapping an element $(a_i)_{i\in I_0\cup I_+}$
of \\ $(\prod_{i\in I_0}\Skew(\End_{\tens{E_i}{k}{D}}(V_i),\sigma))\times (\prod_{i\in I_+}\End_{\tens{E_i}{k}{D}}(V_i))$
to $(\sum_{i\in I_0} a_i) + (\sum_{i\in I_+}(a_i-\sigma(a_i))).$

\begin{assumption}\label{assIntersectionProperty}
For the rest of the section we fix elements $y\in\Building(G)$ and $x\in\Building(H)$ such that 
\[\mf{g}_y\cap \mf{h}=\mf{h}_x.\] The element $x$ is given by a tuple $(x_i)_{i\in I_0\cup I_+}.$
The set $\Lie(\bH_i)((E_i)_0)$ is denoted by $\mf{h}_i.$ For $i\in I_0$, we write 
$\mf{h}_{x_i}$ and $\tilde{\mf{h}}_{x_i}$ for the Lie algebra filtration and the square lattice function of $x_i,$ respectively.
\end{assumption}

\begin{lemma}\label{lemLambdaysplitsUnderVi}
 The lattice function corresponding to $y$ splits under $\{V_i|\ i\in I\}.$
\end{lemma}

\begin{proof}
The assertion is equivalent to the fact that all idempotents $1_i$ are elements of $\ti{\mf{g}}_y(0).$

Case 1: We first consider an index $i\in I_{+}.$ Then $1_i$ is an element of $\mf{h}_{x_i}(0)$ and thus $1_i-1_{-i}$ is an element of
$\mf{g}_y(0)$ by \ref{assIntersectionProperty}. Therefore 
\[1_i+1_{-i}=(1_i-1_{-i})^2\in\ti{\mf{g}}_y(0).\]
Hence $1_i$ and $1_{-i}$ are elements of $\ti{\mf{g}}_y(0),$ since $2$ is
invertible in $o_k.$ 

Case 2: We take an index $i\in I_0$ and we assume that 
$\beta_i$ is not zero. Since $\beta_i$ is skew-symmetric and central, we have, for all
$t\in \bbR$,
\[\beta_i\ti{\mf{h}}_{x_i}(t)=\ti{\mf{h}}_{x_i}(t+\nu(\beta_i))\]
and  
\[\beta_i\mf{h}_{x_i}(t)=\ti{\mf{h}}_{x_i}(t+\nu(\beta_i))\cap
\{w\in\ti{\mf{h}}_i\mid\ \sigma|_{\ti{\mf{h}}_i}(w)=w\}.\]
By the invertibility of $2$ in $o_k$, every element of $\ti{\mf{h}}(t)$ is a sum of a skew-symmetric and a symmetric
element of $\ti{\mf{h}}_{x_i}(t),$ which implies 
\begin{align*}\label{eqnLFIncl}
\ti{\mf{h}}_{x_i}(0)&= \mf{h}_{x_i}(0)+\beta_i\mf{h}_{x_i}(-\nu(\beta_i))\\
&\subseteq \mf{g}_y(0)+ \mf{g}_y(\nu(\beta_i))\mf{g}_y(-\nu(\beta_i))\\
&\subseteq \ti{\mf{g}}_y(0).
\end{align*}
Thus the $i$th idempotent $1_i$ is an element of $\ti{\mf{g}}_y(0).$

Case 3: If there is an index $i_0$ such that $\beta_{i_0}=0,$ 
it is unique by Lemma \ref{lem1betai0}  and the two cases above imply
\[1_{i_0}=1-\sum_{i\neq i_0}1_i\in \ti{\mf{g}}_y(0).\]
\end{proof}
\vspace{1em}

The idea of the proof of case 2 is taken from 
\cite[11.2]{broussousStevens:09}. We define $\bG_i:=\bU(h|_{(V_i+V_{-i})^2}),$ for $i\in I_0\cup I_+.$

\begin{corollary}\label{corPsiI}
For non-negative indices there are elements $y_i\in \Building(\bG_i,k_0),$ 
such that the direct sum of the lattice functions of the $y_i$ is the lattice function of $y.$ 
\end{corollary}

Analogous to the definitions for $G$ we use $\mf{g}_i$ and $\mf{g}_{y_i}$ for  $\Lie(\bG_i)(k_0)$ and the Lie algebra filtration of $y_i,$ respectively.

\begin{lemma}\label{lemIntersectionPropertyComp}
For all $i\in I_0\cup I_+$ we have 
\[\mf{h}_{x_i}=\mf{h}_i\cap\mf{g}_{y_i}.\]
\end{lemma}

\begin{proof}
For $t\in\bbR$ we have
\begin{align*}
\bigoplus_i\mf{h}_{x_i}(t)&= \mf{h}(t)\\
&= \mf{g}_y(t)\cap \mf{h}\\
&= (\mf{g}_y(t)\cap (\bigoplus_i\mf{g}_i))\cap \mf{h}\\
&= (\bigoplus_i\mf{g}_{y_i}(t))\cap (\bigoplus_i\mf{h}_i)\\
&= \bigoplus_i(\mf{g}_{y_i}(t)\cap\mf{h}_i),
\end{align*}
where $i$ runs over $I_0\cup I_+,$ and we obtain the assertion.
\end{proof}
\vspace{1em}

The last two lemmas lead to a factorization of any CLF-map.
More precisely we prove the following proposition.

\begin{proposition}\label{propFactOfCLFMap}
Let $\psi_I$ denote the canonical map from $\Building(\prod_i\bG_i,k_0)$ to $\Building(\bG,k_0)$ which  maps a tuple self-dual lattice functions to its sum. Every CLF-map $j$ from $\Building(H)$ to 
$\Building(G)$ factors under $\psi_I,$ i.e. there is a unique map
\[\tau:\ \Building(H)\ra\Building(\prod_i\bG_i,k_0)\text{ such that }j=\psi_I\circ\tau.\]
The map $\tau$ is
\be
\item a CLF-map,
\item affine if $j$ is affine, and
\item $H$-equivariant if $j$ is $H$-equivariant.
\ee
\end{proposition}

\begin{proof}
The image of $j$ is contained in the image of the injective, affine and $\prod_{i\in
I_0\cup I_+}\bG_i(k_0)$-equivariant map $\psi_I$ by Corollary \ref{corPsiI}. The CLF-property of $\tau$ is a consequence of 
Lemma \ref{lemIntersectionPropertyComp}. This proves the proposition. 
\end{proof}
\section{The case where $I_+$ is empty}

We denote the following algebraic group 
\[\left\{\left(\begin{array}{cc}\lambda&0\\ 0&\lambda^{-1}\end{array}\right),\left(\begin{array}{cc}0&\lambda\\ \lambda^{-1}&0\end{array}\right)|\ \lambda\in \bar{k}^{\times}\right\}\]
by $\bO^{is}_{2}.$ It is a form of $\bO_2$ over the prime field of $k.$ 

\begin{remark}\label{remO2is}
\be 
\item The group $\bU(h)$ is $k_0$-isomorphic to $\bO^{is}_{2}$ if and only if  $D=k=k_0,$ $V$ is two-dimensional over $k$ and isotropic with respect to $h$ and $\sigma$-orthogonal. 
\item The connected component of $\bO^{is}_{2}$ is $k$-isomorphic to $\bG_m,$ implying that $\Building(\bO^{is}_{2},k)$ is affinely isomorphic to $\bbR.$
\item All points of $\Building(\bO^{is}_{2},k)$  have the same Lie algebra filtration. Especially all of its affine endomorphisms are CLF-maps.
\ee 
\end{remark}

The remark forces us to exclude $\bO^{is}_{2}$ from the factors. 

\begin{theorem}\label{thmUniquenessI+empty}
There exists only one CLF-map from $\Building(H)$ to $\Building(G)$ if $I_+$ is empty and no group $\bH_i$ is $(E_i)_0$-isomorphic to $\bO^{is}_{2}.$  
\end{theorem}

For the proof we need the following operations on square matrices.

\begin{definition}\label{defTildeOpSqMatr}
For a square matrix $B=(b_{i,j})\in M_r(D)$ the matrix $\tilde{B}$ is defined to
be $(b_{r+1-j,r+1-i})_{i,j},$ i.e. $\tilde{B}$ is obtained from $B$ by a
reflection along the antidiagonal. We define further
\[B^{\rho}:=(\rho(b_{i,j}))_{i,j}.\]
\end{definition}

\begin{proof}
Applying Lemma \ref{lemIntersectionPropertyComp} we can assume that $E$ is a field. 

Case 1: $\beta$ is not zero. Compare with \cite[11.2]{broussousStevens:09}.
We fix an arbitrary CLF-map $j$ from $\Building(H)$ to $\Building(G)$ and an arbitrary element $x$ of $\Building(H).$
By the same argument as in case 2 of Lemma \ref{lemLambdaysplitsUnderVi} we obtain 
\[\ti{\mf{h}}_x(t)\subseteq \ti{\mf{g}}_{j(x)}(t)\]
for all real numbers $t.$ Theorem \ref{thmBS10.3} implies the uniqueness.  

Case 2: $\beta$ is zero.
If $\sigma$ is of the second type there is a skew-symmetric non-zero element
$\beta'$ in $k$ and we can replace $\beta$ by $\beta'$ and apply case 1. 
Thus we only need to consider the case where $\sigma$ is of the first kind.
We fix a point $y\in\Building(G)$ and fix an apartment 
containing $y.$ This apartment is determined by a Witt decomposition. We choose an adapted basis $(w_i)_{1\leq j\leq m}$ such that the Gram matrix $\Gram_{(w_i)}(h)$ of the $\epsilon$-hermitian form $h$ has the
form 
\begin{displaymath}
\left(\begin{array}{ccc}
0 & M & 0\\
\epsilon M & 0 & 0\\
0 & 0 & N\\
\end{array}\right)
\end{displaymath}
with $M:=\antidiag(1,\ldots,1)$ and a diagonal regular matrix $N.$
From \ref{propNormSplitByOGBasis} we deduce that the self-dual $o_D$-lattice function $\Lambda$
corresponding to $y$ is split by this basis. It is thus described by its
intersections with the lines $w_iD,$
i.e. there are real numbers $a_i$ such that 
\[\Lambda(t)=\bigoplus_iw_i\mf{p}_{D}^{\lceil(t-a_i)d\rceil}.\]
Thus the square lattice function of $y$ in $t$ is 
\[\ti{\mf{g}}_{y}(t)=\bigoplus_{i,j}\mf{p}_{D}^{\lceil(t+a_j-a_i)d\rceil}
E_{i,j}\] 
where $E_{i,j}$ denotes the matrix with a 1 in the intersection of the $i$th row
and the $j$th column and zeros everywhere else. See for example
\cite[I.4.5]{broussousLemaire:02}.

It is enough to show that $\ti{\mf{g}}_y$ is determined by the Lie algebra
filtration $\mf{g}_y.$ Indeed, a class of $o_D$-lattice functions contains at most one self-dual lattice function. Thus the self-dual square lattice function of a point of $\Building(G)$ determines the point uniquely.

The adjoint involution of $h$  
\[B\mapsto B^{\sigma}=\Gram_{(w_i)}(h)^{-1}(B^\rho)^T \Gram_{(w_i)}(h)\] on
$\Matr_m(D)$ has under $(w_i)_{1\leq j\leq m}$ the form 
\begin{displaymath}
\left(\begin{array}{ccc}
B_{1,1} & B_{1,2} & B_{1,3}\\
B_{2,1} & B_{2,2} & B_{2,3}\\
B_{3,1} & B_{3,2} & B_{3,3}\\
\end{array}\right)\mapsto 
\left(\begin{array}{ccc}
\tilde{C}_{2,2} & \epsilon \tilde{C}_{1,2} & \epsilon M C_{3,2}^TN\\
\epsilon \tilde{C}_{2,1} & \tilde{C}_{1,1} & M C_{3,1}^T N\\
\epsilon N^{-1} C_{2,3}^TM & N^{-1} C_{1,3}^T M & N^{-1} C_{3,3}^T N\\
\end{array}\right).
\end{displaymath}
The matrices $B_{1,1},B_{1,2},B_{2,1}$ and $B_{2,2}$ are $r\times r$-matrices
and $C:=B^\rho$ where $r$ is the Witt index of $h.$
By the above calculation we obtain that 
$E_{i,j}^{\sigma}$ is $+E_{i,j},$ $-E_{i,j}$ or $\lambda E_{u,l}$ with
$(i,j)\neq (u,l)$ for some $\lambda\in D^\times.$ From the self-duality of 
$\End(\Lambda)$ and since 2 is invertible in $o_k$ we get:
\[\mf{g}_y(t)\cap k(E_{i,j}-E_{i,j}^{\sigma})
=\mf{p}_k^{\lceil t+a_j-a_i\rceil}(E_{i,j}-E_{i,j}^{\sigma}).\]
For the calculation see Lemma \ref{lemIntersectionGaussBracket} below. Thus we can obtain the ex\-po\-nent
$a_j-a_i$ from the know\-ledge of the Lie algebra fil\-tration if
$E_{i,j}$ is not fixed by $\sigma.$ We now con\-si\-der two ca\-ses.

Case 2.1: We assume that there is an anisotropic part in the Witt
de\-com\-po\-si\-tion, i.e. $N$ occurs. The matrix $E_{i,m}$ is fixed by
$\sigma$ if and only if $i$ equals $m.$
Thus from the know\-ledge of the Lie al\-ge\-bra fil\-tra\-tion we know all
dif\-fe\-ren\-ces $a_i-a_{m}$ for all in\-de\-xes $i$ dif\-fe\-rent
from $m,$ and thus by sub\-trac\-tions we know the dif\-fe\-ren\-ces
$a_i-a_j$ for all $i$ and $j.$

Case 2.2: Now we assume that there is no an\-iso\-tro\-pic part in the Witt
de\-com\-po\-si\-tion. If $\epsilon$ is $-1,$ no $E_{i,j}$ is fixed and we can
obtain the differences $a_i-a_j$ for all $i$ and $j.$ As a
consequence, we only have to consider the case where $h$ is hermitian and $D=k$ 
(see \cite[1.14]{bruhatTitsIV:87}).
Here the matrix $E_{i,j}$ is fixed by $\sigma$ if and only if $i+j=m+1.$ Thus we
can determine all differences $a_i-a_j$ where $i+j\neq m+1.$ If $m$ is
at least $4$ for an index $i$ there is an index $k\neq i$ with $i+k\neq m+1$ and
we can obtain $a_i-a_{m+1-i}$ if we add
$a_k-a_{m+1-i}$ to $a_i-a_k.$
If $m$ equals $2$, then the group
$\bG$ is 
$k$-isomorphic to $\bO^{is}_{2}$ which is excluded by the assumption of the
theorem. 
\end{proof}

The idea of taking the root system of $\bG$ over $k$ for the non-unitary step of the last proof was given by P. Broussous. To complete the proof we need the following lemma.

\begin{lemma}\label{lemIntersectionGaussBracket}
For all $t\in\bbR$ we have
\[\mf{p}_D^{\lceil td\rceil}\cap k=\mf{p}_k^{\lceil t\rceil}.\]
\end{lemma}

\begin{proof}
For an element $x$ of $k$, we have $x\in\mf{p}_D^{\lceil td\rceil}$ if and only if $\nu(x)\geq\frac{\lceil td\rceil}{d}$.
There are integers $l$ and $s$ such that $1\leq s\leq d$ and
\[{\lceil td\rceil}=ld+s\]
Thus ${\lceil t\rceil}= l+1$ and we get that $\nu(x)\geq\frac{\lceil td\rceil}{d}$ if and only if $\nu(x)\geq {\lceil t\rceil}$.
The  ''only if'' follows from $\nu(x)\in\bbZ.$
\end{proof}

\begin{remark}
\be
\item In particular the proof of Theorem \ref{thmUniquenessI+empty} shows that, if $\bH_i$ is $(E_i)_0$ isomorphic to $\bO^{is}_{2},$ then $\beta_i$ has to be zero and $\sigma$ is of the first kind.
\item For positive indices $\bH_i$ is not isomorphic to $\bO_2$ because the latter is not connected.
\item Let us assume that $\beta$ is separable. The above remarks and \ref{remO2is} imply, for $i\in I_0\cup I_+$, that $\bO_{2}^{is}$ is $k_0$-isomorphic to $\Res_{(E_i)_0| k_0}(\bH_i)$ if and only if $\bH_i$ is $(E_i)_0$-isomorphic to $\bO_{2}^{is}.$  
\ee
\end{remark}

\section{Rigidity of Euclidean buildings}

To have an approach to a uniqueness statement if there are no restrictions on $I$ we show that Euclidean buildings are rigid for functionals.

\begin{definition}
\be
\item A {\it set with affine structure} is a pair $(S,*)$ consisting of a non-empty set 
$S$ and a map \[*:[0,1]\times S\times S\ra S,\] which we denote
\[ts_1+(1-t)s_2:=*(t,s_1,s_2).\]
\item  An \textit{affine functional} $f$ on a set 
$S$ with affine structure is an affine map from $S$ to $\bbR,$ i.e.
\[f(tx+(1-t)y)=tf(x)+(1-t)f(y),\]
for all $t\in [0,1]$ and $x,y\in S.$
\ee
\end{definition}

\begin{proposition}\label{propRigidityForAffineBuildings}
Let $\Omega$ be a thick Euclidean building and $|\Omega|$ be its geometric
realization. Then every affine functional $a$ on $|\Omega|$ is constant.
\end{proposition}

For the definition of a thick Euclidean building and its geometric realization
see \cite[VI.3]{brown:89}.

\begin{proof}Let $C_1,\ C_2$ and $C_3$ be three pairwise different adjacent chambers having a common co-dimension 1 face $S.$ We denote by $P_i$ the unique vertex of $C_i$ which is not a vertex of $S.$ 
The line segment $[P_1,P_2]$ meets $[P_1,P_3]$ and $[P_2,P_3]$ in a point $Q\in
|\bar{S}|.$ This can be seen as follows.  
We are working in three different apartments simultaneously. If $\Delta_{ij}$
denotes an apartment 
containing $C_i$ and $C_j,$  for different $i$ and $j,$ the affine isomorphism 
from $|\Delta_{12}|$ to $|\Delta_{13}|$ fixing $|\Delta_{12}\cap\Delta_{13}|$ sends
$[P_1,P_2]$ to $[P_1,P_3]$ and 
thus the unique intersection point in $[P_1,P_2]\cap |\bar{C_1}|\cap |\bar{C_2}|$
lies on $[P_1,P_3],$ and similarly on $[P_3,P_2].$  
Without loss of generality assume that $a(Q)$ vanishes. If $a(P_1)$ is negative
then $a(P_2)$ and $a(P_3)$ are positive by the affineness of $a.$ Thus 
$a(Q)$ is positive since it lies on $[P_2,P_3].$ A contradiction. Using
galleries we obtain that $a$ is constant on vertices of the same type. An
apartment is affinely generated by its vertices of a fixed type. Thus $a$ is
constant on every apartment and therefore on $|\Omega|,$ since any two apartments are connected by a gallery. 
\end{proof}

\begin{proposition}\label{propRigidityForClassicalGroups}
If $\bG$ is not $k_0$-isomorphic to $\bO_{2}^{is}$, then every affine functional on $\Building(G)$ is constant.
\end{proposition}

\begin{proof} Not being $k_0$-isomorphic to $\bO_{2}^{is},$ the unitary group $\bU(h)$ has no $k_0$-rational characters on the 
connected component of the identity implying that the non-extended and the extended buildings are equal (see \ref{propExEnlargedBTBClassGrps}). 
If $\bG$ is totally isotropic, then $\Building(G)$ is a point and otherwise
it is the geometric realization of a thick Euclidean building. Now we apply
Proposition \ref{propRigidityForAffineBuildings}.  
\end{proof}

\begin{proposition}\label{propRigidityForGLnSO2}
\be
\item A $k^\times$-invariant affine functional on $\Building(\bGL_D(V),k),$ i.e. on $\Lattone{o_D}{V},$ is constant. 
\item Every $k^\times$-invariant affine functional on
$\Building(\bO_{2}^{is},k)$ is constant. 
\ee
\end{proposition}

\begin{proof}
\be
\item Every fiber of a $k^\times$-invariant affine functional on 
$\Building(\bGL_D(V),k)$ is a union of classes of $o_D$-lattice functions.
It therefore factorizes to an affine functional on $\building(\bGL_D(V),k).$ Now we apply Proposition \ref{propRigidityForAffineBuildings}.
\item This follows from part 1, because $\Building(\bO_{2}^{is},k)$ is isomorphic to $\Building(\bG_m,k)$ via a
$k^\times$-equivariant affine bijection. 
\ee
\end{proof}
\section{The general case}

We introduce the notion of a translation in order to understand the class of affine  CLF-maps which satisfy a mild equivariance condition.

\begin{definition}
\be
\item A {\it translation}  of $\Lattone{o_D}{V}$ is a map from $\Lattone{o_D}{V}$ to itself 
of the form 
\[\Lambda\mapsto\Lambda+s\]
where $s$ is a real number.
\item A {\it translation} of $\Lattone{h}{V}$ is the identity if $\bG$ is not $k_0$-isomorphic to $\bO^{is}_{2}.$
\item A {\it translation} of $\Building(\bH,k_0)$ is a product of translations of the factors. 
\ee 
\end{definition}

\begin{remark}\label{remOis2k_0bbR}
A translation of $\Building(\bO_{2}^{is},k_0)$ is included in the definition because 
 $(\bO_{2}^{is})^0$ is $k_0$-isomorphic to $\bG_m,$ i.e. there is a natural bijection
from $\Building(\bO_{2}^{is},k_0)$ to $\Lattone{k_0}{k_0}.$
\end{remark}

Let us fix a map $j$ from $\Building(H)$ to $\Building(G)$ constructed
as in the proof of Theorem \ref{thmExistence}. In this section we are going to prove:

\begin{theorem}\label{thmCLFUnitaryCase}
If $\phi$ is an affine and $\Centr(\prod_i\bH_i^0((E_i)_0))$-equivariant CLF-map from $\Building(H)$ to $\Building(G)$ then $j^{-1}\circ\phi$ is a
translation of $\Building(H).$ In terms of lattice functions, the image of
$\phi$ is the set of self-dual $o_E$-$o_D$-lattice functions in $V$ and $\phi$
is $\prod_i\bH_i^0((E_i)_0)$-equivariant. 
\end{theorem}

The composition of $j^{-1}$ with $\phi$ is possible by the following fact.

\begin{proposition}\label{propCLFETFP}
The image of a CLF-map from $\Building(H)$ to $\Building(G)$ is a
subset of the set of $o_E$-$o_D$-lattice functions.
\end{proposition}

\begin{proof}
By Lemma \ref{lemIntersectionPropertyComp} we can assume that 
\[I_0\cup I_+=\{i\}.\] 

Case 1: ($\beta=0$) The field $E$ is $k$ and there is nothing to prove.

Case 2: ($i\in I_0$ and $\beta\neq 0$) There is only one CLF-map by Theorem \ref{thmUniquenessI+empty}
and it fulfills the assertion by Theorem \ref{thmExistence}.

Case 3: ($i\in I_+$) We choose two arbitrary points $y\in\Building(G)$ and $x\in\Building(H)$ such that $\mf{g}_y\cap \mf{h}=\mf{h}_x.$ The lattice function $\Lambda$ of $y$ splits under
$(V_i,V_{-i})$ by corollary \ref{corPsiI}. By self-duality it is sufficient 
to prove that $\Lambda\cap V_i$ is an $o_{E_i}$-lattice function. The building
\[\Building(H)=\Building(\bGL_{\tens{E_i}{k}{D}}(V_i),E_i)\]
is identified with the set of lattice functions over a skew field whose center is
$E_i.$ Thus we get
\bi
\item $(a-a^{\sigma})\in\mf{g}_y(0)$ for all $a\in
o_{E_i}^{\times},$
\item $\pi_{E_i}-\pi_{E_i}^{\sigma}\in \mf{g}_y(\frac{1}{e}),$ and 
\item $\pi_{E_i}^{-1}-(\pi_{E_i}^{-1})^{\sigma}\in
\mf{g}_y(-\frac{1}{e}),$
\ei
where $e$ is the ramification index of $E_i|k$ and $\pi_{E_i}$ is a prime
element of $E_i.$ We conclude that $1_i\Lambda$ is an $o_{E_i}$-lattice
function.
\end{proof}

\begin{proof}[of Theorem \ref{thmCLFUnitaryCase}]
Proposition \ref{propCLFETFP} enables us to define 
\[\tau:=j^{-1}\circ\phi\]
because the image of $j$ is the set of all $o_E$-$o_D$-lattice functions in $V.$

If we know that $\tau$ is a translation, then all assertions of Theorem \ref{thmCLFUnitaryCase} hold: A translation is a bijection and we conclude that $\phi$ and $j$ have the same image. 
The $\prod_i\bH_i^0((E_i)_0)$-equivariance of $\phi$ follows because $j$ and $\tau$ are $\prod_i\bH_i^0((E_i)_0)$-equivariant. 

We denote the coordinates of $\tau$ by $\tau_i,$ $i\in I_0\cup I_+.$ We need two 
steps to show that $\tau$ is a translation. Let us fix $i\in I_0\cup I_+.$

Step 1: We prove that  the coordinate $\tau_i$ only depends on $x_i.$

Case 1.1: $i\in I_0,$ such that $\bO_{2}^{is}$ is not $(E_i)_0$-isomorphic to $\bH_i$. By Theorem \ref{thmUniquenessI+empty} we have $\tau_i(x)=x_i,$ for all $x\in \Building(H).$  

Case 1.2: We assume that we have an index $i\in I$ such that $\bH_i$ is
$(E_i)_0$-isomorphic to $\bO_{2}^{is}$, whose building is affinely isomorphic to $\bbR$ by Remark \ref{remOis2k_0bbR}.
 If we fix an index $t\in (I\cup I_+)\setminus\{i\}$ and coordinates $x_l$ for $l\in
 (I\cup I_+)\setminus\{t\}$, then the map 
\[x_t\mapsto \tau_i(x)\]
is constant by Proposition \ref{propRigidityForClassicalGroups} or
\ref{propRigidityForGLnSO2}. Thus $\tau_i$ does not depend on $x_t.$

Case 1.3: $i\in I_+.$ Let us fix $x\in \Building(H)$. The lattice functions $\Lambda_{\tau_i(x)}$ and $\Lambda_{x_i}$ are equivalent by the CLF-property of $\tau$. We define $a_i(x)$ to be the real number such that 
\[\Lambda_{\tau_i(x)}=\Lambda_{x_i}+a_i(x).\]
The map $a_i$ is affine, since $\tau_i$ is. By an analogous argument as in 
case 1.2 we have that $a_i$ does not depend on the $t$th coordinate, for $t\in (I_+\cup I_0)\setminus\{i\}$. 

Step 1 allows us to define a map $\tilde{\tau}_i$ from $\Building(\bH_i,(E_i)_0)$
to itself by
\[\tilde{\tau}_i(x_i):=\tau_i(x),\ x\in \Building(H).\]

Step 2: Here we show that $\tilde{\tau}_i$ is a translation. 
We firstly consider an index $i\in I_0$ such that $\bH_i$ is $(E_i)_0$-isomorphic to $\bO_{2}^{is}.$ In this case we have $k=k_0$ and $E_i=1_ik.$ We identify $\Building(\bO_{2}^{is},k)$ with $\bbR.$ The $\bSO_{2}^{is}(k)$-equivariance of $\tau_i$
gives
\[\tilde{\tau}_i(x_i+1)=\tilde{\tau}_i(x_i)+1.\] 
The affineness property implies that $\ti{\tau}_i$ is a translation.
For $i\in I_+$ the map $a_i$ in case 3 of step 1 is an affine
functional and the $k^{\times}$-equivariance of $\tau_i$ implies the
$k^{\times}$-invariance of $a_i,$ because one gets in terms of lattice functions
\begin{align*}
\Lambda\pi_k+a_i(\Lambda\pi_k)&= \ti{\tau}_i(\Lambda\pi_k)\\
&= (\ti{\tau}_i(\Lambda))\pi_k\\
&= \Lambda+a_i(\Lambda)-1\\
&= \Lambda\pi_k+a_i(\Lambda),
\end{align*}
where $\pi_k$ is a prime element of $k$.
Thus $a_i$ is constant by Proposition \ref{propRigidityForGLnSO2}.
\end{proof}

\section{Torality}

In this section we want to prove that the map constructed in section \ref{secCLFExistence} respects apartments and is  toral if $\beta$ is separable.

\begin{definition}
A map 
\[f:\ \Building(G_1,k_0)\ra\Building(G_2,k_0)\]
between two extended buildings of reductive groups defined over $k_0$ is called \textit{toral}\/ if, for each maximal $k_0$-split
torus $S$ of $G_1$, there is a maximal $k_0$-split torus $T$ of $G_2$ containing $S$ such that
$f$ maps the apartment corresponding to $S$ into the apartment corresponding to $T.$ 
An analogous definition applies to maps between non-extended buildings.
\end{definition}

Maximal $k_0$-split tori of $\bG$ can be characterized in terms of Witt decompositions. Given a Witt decomposition  $\{W_l\mid l\in L\}$ of $V$, there is exactly one maximal $k_0$-split torus $T$ of $\bG$ which satisfies 
\[\Centr_{\bG}(T)(k_0)=\{g\in\bG(k_0)\mid g(W_l)\subseteq W_l,\ l\in L,\ (g-\id_V)(W_0)\subseteq \{0\}\}.\] We recall, that 
$W_0$ is the orthogonal complement of the sum of all $W_l.$
Every maximal $k_0$-split torus arises in this way, because they are conjugate to each other by elements of $\bG(k_0).$
The $D$-vector spaces $W_l,\ l\in L,$ are exactly the irreducible $T(k_0)$-invariant $D$-subspaces 
of the orthogonal complement of 
\[\{v\in V\mid\ t(v)=v, t\in T(k_0)\}.\]
Analogously one shows that the set of maximal $k$-split tori of $\bGL_D(V)$ is 
one-to-one correspondence with the set of all decompositions of $V$ into one-dimensional $D$-subspaces. 

\begin{lemma}\label{lemPsiItoral}
The map $\psi_I$ from $\Building(\prod_i\bG_i,k_0)$ to $\Building(\bG,k_0)$ defined
in Proposition \ref{propFactOfCLFMap} by 
\[\psi_I((\Lambda_i)_{i\in I_0\cup I_+}):=\bigoplus_{i\in I_0\cup I_+}\Lambda_i\]
is toral. 
\end{lemma}

\begin{proof}
For $i\in I_0\cup I_+,$ let $S_i$ be a maximal $k_0$-split torus of $\bG_i$ and 
$\{W_l^i\mid\ l\in L_i\}$ be the corresponding Witt decomposition of $V_i+V_{-i}.$
Further let $\Delta_i$ be the apartment of $S_i$ in $\Building(\bG_i,k_0).$ Let $\alpha_i$ be an element of $\Delta_i$ seen as a self dual $D$-norm on $V_i+V_{-i}.$  By \cite[2.9]{bruhatTitsIV:87} the norm 
\[\alpha:=\sum_i\alpha_i\]
has the form
\[\alpha(\sum_i w_0^i)=\frac12\inf_i(\nu(h(w_0^i,w_0^i)))\]
on $V_a:=\sum_iW_0^i,$ i.e. this form does not depend on $(\alpha_i)_i.$
Let us take a splitting Witt decomposition 
$\{W^a_l\mid\ l\in L_a\}$ of $V_a$ for the restriction of $\alpha$
to $V_a.$ Then the torus $\prod_i S_i$ is a subtorus of 
the maximal $k_0$-split torus $T,$ which corresponds to the Witt decompostion 
\[\bigcup_{i\in I_0\cup I_+\cup\{a\}}\{W_l^i\mid l\in L_i\}\] of $V,$ and $\psi_I$ maps 
$\prod_i\Delta_i$ into the apartment of $T.$
\end{proof}

\begin{proposition}\label{propToralityOfj}
The map $j$ constructed in the proof of Theorem \ref{thmExistence} maps apartments into apartments. Furthermore, $j$ is toral if $\beta$ is separable.
\end{proposition}

Without loss of generality we may assume that $I_0\cup I_+=\{i\}$ by Lemma \ref{lemPsiItoral}.

\begin{proof}[$i\in I_+$]
The map $\ti{j}^{E_i}$ of Theorem \ref{thmBL02II.3.1II.4} from $\Building(H)$ to 
\[\Building(\bGL_D(V_i),k)=\Building(\Res_{k|k_0}(\bGL_D(V_i)),k_0)\]
maps apartments into apartments and is toral if $\beta_i$ is separable by \cite[5.1]{broussousLemaire:02}.
We prove that the canonical map $\phi$, see (\ref{eqEmbedBuild}), from $\Building(\Res_{k|k_0}(\bGL_D(V_i)),k_0)$ to $\Building(\bG,k_0)$ defined by
\[\Lambda\in\Lattone{o_D}{V_i}\mapsto (\Lambda\oplus\Lambda^{\#_{-i}})\in\Lattone{h}{V}\]  is toral.
A maximal $k$-split torus $S$ of $\bGL_D(V_i)$ corresponds to a decomposition of $V_i$ in one-dimensional $D$-subspaces, i.e. there is a decomposition $V_i=\bigoplus_lV_{i,l}$ such that 
\[ S(k)=\Centr(\{g\in\GL_D(V_i)|\ g(V_{i,l})\subseteq V_{i,l} \text{ for all } l\}).\]
Let $V_{-i,j}$ be the subspace of $V_{-i}$ dual to $V_{i,j},$ i.e.
\[V_{-i,j}:=\{v\in V_{-i}\mid\ h(v,V_{i,k})=\{0\}, k\ne j\},\]
 and let $T$ be the torus given by the decomposition 
\[V=\bigoplus_l(V_{i,l}\oplus V_{-i,l}).\]
Under the map (\ref{eqEmbedGroup}), from $\Aut_D(V_i)$ to $\Aut_D(V_i)\oplus\Aut_D(V_{-i})$ defined by 
\[g\in\Aut_D(V_i)\mapsto (g,0)+\sigma((g^{-1},0))\in\Aut_D(V_i)\oplus\Aut_D(V_{-i})\]
the set $S(k)$ is mapped into $T(k)$ and under $\phi$ the apartment of $S$ is mapped into the apartment of $T.$
Let $S'$ and $T'$ be the maximal $k_0$-split 
sub-tori of $\Res_{k|k_0}(S)$ and $\Res_{k|k_0}(T)$ respectively. The set $S'(k_0)$ is mapped into \mbox{$(T'\cap\bG)(k_0)$} under 
(\ref{eqEmbedGroup}).  The image of 
$\phi$ only consists of self-dual lattice functions. Hence 
$\phi$ seen as a map 
from $\Building(\Res_{k|k_0}(\bGL_D(V_i)),k_0)$ to $\Building(\bG,k_0)$ is toral.  
\end{proof}

From now on we assume $i\in I_0.$ Here we need the notion of tori adapted to a decomposition of $V.$

\begin{definition}
Assume we are given a decomposition 
\[V=V'^+\oplus V'^-\oplus V'^{0}\]
such that $V'^+$ and $V'^-$ are maximal totally isotropic and $V'^+\oplus V'^-$ is orthogonal to $V'^{0}$ with respect to $h.$
A maximal $k_0$-split torus $T$ of $\bG$ is \textit{adapted to } $(V'^+,V'^-,V'^0)$ if there is a Witt decomposition $\{W_l|\ l\in L\}$ corresponding to $T$ with anisotropic part $V'^0$ such that 
\[\bigoplus_l (W_l\cap V'^+)=V'^+\text{ and } \bigoplus_l (W_l\cap V'^-)=V'^-.\]
We say that an apartment of $\Building(G)$ is \textit{adapted} to $(V'^+,V'^-,V'^0)$ if every lattice function in this apartment is split by $(V'^+,V'^-,V'^0).$
\end{definition}

\begin{proof}[$i\in I_0$]
We have $E=E_i.$ There are a central division algebra $\Delta$ over $E$ and
 a finite-dimensional right vector space $W$ such that  $\End_{\tens{E}{k}{D}}(V)$ is $E$-algebra isomorphic 
to $\End_{\Delta}(W).$ We identify the $E$-algebras $\End_{\tens{E}{k}{D}}(V)$ and $\End_{\Delta}(W)$ via a fixed isomorphism and we fix a signed hermitian form $h_E$ which corresponds to the restriction $\sigma_E$ 
of $\sigma$ to the $E$-algebra $\End_{\Delta}(W).$ Let $r$ be the Witt index of $h_E.$ We fix a decomposition 
\beq\label{eqDirSumW}
W=(W^+ \oplus W^{-})\oplus W^0
\eeq
such that $W^+$ and $W^{-}$ are maximal isotropic subspaces of $W$ contained in the orthogonal complement of $W^0.$
 Let $e_+, e_-$ and $e_0$ be the projections to the vector spaces $W^+,W^-$ and $W^0$ via the direct sum (\ref{eqDirSumW}). We define 
\[ V^+:=e^+V,\ V^-:=e^-V\text{ and }V^0:=e^0V.\]
Consider the following diagram.

\[\begin{matrix}
\Building(H) & \leftarrow  &\Building(\bU((h_E)|_{W^0\times W^0}), E_0)\times\Building(\bGL_{\Delta}(W^+),E) & \ra &\building(\bGL_{\Delta}(W^+),E)  \\
\downarrow j &  & \downarrow \alpha & & \da \\
\Building(G) & \leftarrow & \Building(\bU(h|_{V^0\times V^0}),k_0)\times\Building(\bGL_D(V^+),k) &\ra & \building(\bGL_D(V^+),k)\\
\end{matrix}\]
where the vertical maps are induced by $j.$ The right horizontal maps 
map a pair $(x,\Lambda)$ to the class of $\Lambda.$ The right vertical arrow satisfies the CLF-property and its image only consists of $E^{\times}$-fixed points of $\building(\bGL_D(V^+),k),$ both properties inherited from $j.$ Thus the map in the right column is $j^E,$ i.e. the inverse of $j_E,$ because otherwise we could construct a CLF-map from 
$\building(\bGL_D(V^+),k)^{E^{\times}}$ to $\building(\bGL_{\Delta}(W^+),E)$ different from $j_E,$ but such a CLF-map is unique by \cite[II.1.1.]{broussousLemaire:02}. Now $j^E$ maps apartments into apartments which implies that $j$ maps apartments adapted to $(W^+,W^-,W^0)$ into apartments adapted to $(V^+,V^-,V^0).$ 

We now prove that $j$ is toral if $E|k$ is separable. Let us assume that $E|k$ is separable. This implies that the right column $j^E$ is toral by \cite[5.1]{broussousLemaire:02} which implies the torality of $\alpha$ because the only maximal $E_0$-split torus of the anisotropic group $\bU((h_E)|_{W^0\times W^0})$ is the trivial group. The torality of $\alpha$ implies the torality of $j$ on tori adapted to $(W^+,W^-,W^0).$ Hence $j$ is toral because the triple $(W^+,W^-,W^0)$ was chosen arbitrarily.
\end{proof}
\appendix
\section*{Appendix}
\section{The centralizer of a separable Lie algebra element}\label{appendixCentralizer}

We prove the representation of the centralizer as a product of general linear groups and Weil restrictions of unitary groups if $\beta$ is separable.
We still rely on the notation from section \ref{secNotation}. For more details about the Weil restriction we recommend \cite[1.3.]{weil:82} and \cite[20.5.]{knus:98}.
The definitions of $D'_i$ and $V'_i$ are given after notation \ref{notActionOfSigmaOnI}.

\begin{proposition}\label{propResScalars}
Let us assume that $\beta$ is separable. The centralizer $\Centr_{\bG}(\beta)$ is $k_0$-isomorphic to 
\[\prod_{i\in I_0}\Res_{(E_i)_0|k_0}(\bU(\sigma|_{\End_{\tens{E_i}{k}{D}}(V_i)}))
\times\prod_{i\in I_+}\Res_{(E_i)_0|k_0}(\bGL_{D'_i}(V'_i));\] 
in particular it is reductive and defined over $k_0.$ 
\end{proposition}

Before we come to the proof, we need some preparations on restriction of scalars.

\begin{lemma}\label{lemResScalarsFromTensorProduct}
 Let $D'$ be a skew field of finite index such that the center, denoted by $E,$ is a non-Archimedean local field.  Let $V'$ be a finite-dimensional 
right $D'$-vector space. Assume further that $\sigma'$ is an involution on $\End_{D'}(V')$ whose set of fixed points
in $E$ is $E_0.$ Let $k_0$ be a subfield of $E_0$ such that $E_0|k_0$ is separable and finite. Then $\bU(\tens{\sigma'}{k_0}{\id_{\Omega}})$ is $k_0$-isomorphic to $\Res_{E_0|k_0}(\bU(\sigma'))$ 
\end{lemma}

\begin{remark}
If $\bV$ is an affine variety defined over $E_0$ there is an isomorphism
\[\Res_{E_0|k_0}(\bV)\cong \prod_{\gamma}\bV^{\gamma},\] 
defined over the normal hull of $E_0|k_0$, where $\gamma$ passes through the set $\Gamma$ of all field embeddings $E_0\hookrightarrow \Omega$ fixing $k_0$, and one obtains $\bV^{\gamma}$ from $\bV$ in the following way. We choose an automorphism $\bar{\gamma}$ of $\Omega$
whose restriction to $E_0$ is $\gamma.$ Let $P_1,\ldots,P_l$ be polynomials in $n$ variables with coefficients in $E_0$ such that 
\[\bV:=\{x\in\Omega^{t}|\ P_i(x)=0,\ i=1,\ldots,l\},\] for some $t\in\bbN.$
One defines 
\[\bV^{\gamma}:=\{(\bar{\gamma}(x_j))_{1\leq j\leq t}|\ x\in\bV\}
=\{x\in\Omega^t |\ P_i^{\gamma}(x)=0,\ i=1,\ldots,l\},\]
where $P_i^{\gamma}$ is the polynomial obtained from $P_i$ by applying $\gamma$ to the coefficients.
\end{remark}

\begin{proof}[of Lemma \ref{lemResScalarsFromTensorProduct}]
The $\Omega$-algebra $\tens{\End_{D'}(V')}{k_0}{\Omega}$ is canonically isomorphic to 
$\tens{\End_{D'}(V')}{E_0}{\tens{E_0}{k_0}{\Omega}}$ and  $\tens{E_0}{k_0}{\Omega}$ is isomorphic to 
$\Omega^{[E_0:k_0]}$ via
\[\tens{e}{k_0}{w}\mapsto (\gamma(e)w)_{\gamma\in\Gamma}.\] 
We denote by $\Omega^{\gamma}$ the left $E_0$-vector space $\Omega$ under the action
\[e.\omega:=\gamma(e)\omega.\]
Thus 
\[\tens{\End_{D'}(V')}{E_0}{\tens{E_0}{k_0}{\Omega}}\cong \prod_{\gamma}\tens{\End_{D'}(V')}{E_0}{\Omega^{\gamma}}.\]
The fact that $\sigma'$ fixes $E_0$ implies that $\tens{\sigma'}{k_0}{\id_{\Omega}}$ is the product of involutions
$\tens{\sigma'}{E_0}{\id_{\Omega^{\gamma}}}.$ It is enough to prove 
\[\bU(\tens{\sigma'}{E_0}{\id_{\Omega^{\gamma}}})=\bU(\tens{\sigma'}{E_0}{\id_{\Omega}})^{\gamma}.\]
To show the latter, we fix $\gamma$ and we choose an extension $\bar{\gamma}$ to $\Omega.$ The ring isomorphism
\[\Phi:\ \tens{\End_{D'}(V')}{E_0}{\Omega}\ra  \tens{\End_{D'}(V')}{E_0}{\Omega^{\gamma}}\]
sending $\tens{f}{E_0}{\omega}$ to $\tens{f}{E_0}{\bar{\gamma}(\omega)}$ 
satisfies 
\[\Phi\circ (\tens{\sigma'}{E_0}{\id_{\Omega}})=(\tens{\sigma'}{E_0}{\id_{\Omega^{\gamma}}})\circ\Phi.\]
Thus an element $g$ of $\tens{\End_{D'}(V')}{E_0}{\Omega}$ lies in $\bU(\tens{\sigma'}{E_0}{\id_{\Omega}})$ if and only 
if $\Phi(g)$ lies in $\bU(\tens{\sigma'}{E_0}{\id_{\Omega^{\gamma}}})$, which proves the lemma. 
\end{proof}

\begin{proof}[of Proposition \ref{propResScalars}]
Without loss of generality we assume that $I_0\cup I_+=\{i\}.$ 

Case 1: $I_0=\{i\}.$ We fix an algebraic closure $\Omega$ of $E.$ We have
\begin{align*} 
\Centr_{\tens{\End_D(V)}{k_0}{\Omega}}(\tens{\beta}{k_0}{1})& \cong\ \tens{\Centr_{\End_D(V)}(\beta)}{k_0}{\Omega}\\
&\cong\ \tens{\End_{D'_i}(V'_i)}{k_0}{\Omega}.
\end{align*}
The involution defining $\bU(\sigma)$ on  $\tens{\End_{D}(V)}{k_0}{\Omega}$ is 
$\tens{\sigma}{k_0}{\id}.$ The above $\Omega$-algebra isomorphism defines an involution $\tens{\sigma'_i}{k_0}{\id}$
on the right side of the equation where $\sigma'_i$ is an involution on $\End_{D'_i}(V'_i)$ whose set of fixed points in
$E$ is $E_0.$ By Lemma \ref{lemResScalarsFromTensorProduct}
 the group $\U(\tens{\sigma'_i}{k_0}{\id})$ is the Weil restriction of $\bU(\sigma'_i)$
from $E_0$ to $k_0.$ 

Case 2: $I_+=\{i\}$. The algebra $E$ is a direct product of two fields $E_i$ and $E_{-i}$ and we define
$E_0$ to be the set of fixed points of $\sigma$ in $E.$ In the following part of the proof
we use the bijections
\[E_i\ra E_0\ra E_{-i},\ e_i\mapsto e_i+\sigma(e_i)\mapsto \sigma(e_i).\]
We fix an algebraic closure $\Omega$ of $E_0.$
As in the proofs of Lemma \ref{lemResScalarsFromTensorProduct} and  case 1, we get the product decomposition
\begin{align*} 
\Centr_{\tens{\End_D(V)}{k_0}{\Omega}}(\tens{\beta}{k_0}{1})& \cong\ \tens{\Centr_{\End_D(V)}(\beta)}{k_0}{\Omega}\\
&\cong\ \tens{(\End_{D'_i}(V'_i)\oplus (\End_{D'_i}(V'))^{op})}{k_0}{\Omega}\\
&\cong\ \prod_{\gamma}(\tens{(\End_{D'_i}(V'_i)\oplus (\End_{D'_i}(V'_i))^{op})}{E_0}{\Omega^{\gamma}})\\
&\cong\ \prod_{\gamma}((\tens{\End_{D'_i}(V'_i)}{(E_i)_0}{\Omega^{\gamma}})\oplus (\tens{\End_{D'_i}(V'_i)}{(E_{-i})_0}{\Omega^{\gamma}})^{op}).\\
\end{align*}
The unitary group 
\[\{g\in (\tens{\End_{D'_i}(V'_i)}{E_i}{\Omega})\oplus (\tens{\End_{D'_{-i}}(V'_{-i})}{(E_{-i})_0}{\Omega})^{op}|\ 
g(\tens{\sigma}{E_0}{\id_{\Omega}})(g)=1\}\]
is $E_i$-isomorphic to $\bGL_{D'_i}(V'_i).$ We conclude as in the proof of Lemma \ref{lemResScalarsFromTensorProduct}.
\end{proof}

\section{Rational characters}
In this section we show that the extended and the non-extended buildings of $G$ are equal 
in almost all cases. The following statement is common knowledge, but we give a prove 
for the sake of completeness. Let $X^*(?)_{k_0}$ denote the set of $k_0$-rational characters of $?$. Let us recall that by definition $\Building(G)$ is different from $\building(G)$ if and only if the set $X^*(\bG^0)_{k_0}$ is non-trivial (see \cite[4.2.16]{bruhatTitsII:84}). 

\begin{definition}
The {\it special unitary group} $\bSU(h)$ of $h$ is a connected reductive group defined over $k_0,$ whose set of rational points, which we denote by $\SU(h)$, is the intersection of $\bU(h)(k_0)$ with the kernel of the reduced norm on $\End_D(V).$
\end{definition}

\begin{remark}\cite[2.15]{platonovRapinchuk:94}\label{remConnectedCompG}
\be 
\item If $\sigma$ is a unitary involution, then $\bG$ is a $k$-form of $\bGL_{md}(\bar{k}).$
\item If $\sigma$ is symplectic, then $\bG$ is $k$-isomorphic to $\bSU(h)$ which is a $k$-form of the symplectic group $\bSp_{md}(\bar{k}),$ in particluar $\bG$ is semisimple.
\item If $\sigma$ is orthogonal, then $\bG^0$ is $k$-isomorphic to $\bSU(h)$ which is a $k$-form of the special orthogonal group $\bSO_{md}(\bar{k}),$ in particular $\bG^0$ is semisimple if $md\neq 2.$ 
\ee
\end{remark}

We further need the following theorem.

\begin{proposition}\label{propExEnlargedBTBClassGrps}
$X^*(\bG^0)_{k_0}\neq 1$ if and only if 
\beq\label{eqCondO2is}
m=2 \text{ and } d=1 \text{ and } \sigma \text{ is orthogonal and } h \text{ is isotropic. }
\eeq
\end{proposition}

\begin{lemma}\label{lemMatrixEmbedding}
Let $L|L'$ be a field extension and $n\in\mathbb{N}.$ Let $\bar{L}$ be an algebraic closure of $L.$ 
Let $\phi$ be an $L'$-algebra monomorphism from $\Matr_n(L)$ into $\Matr_n(\bar{L}).$ 
Then there is an element $\psi$ in $\Gal(\bar{L}|L')$ inducing an $L'$-algebra automorphism of $\Matr_n(\bar{L})$ via
\[\Psi((a_{ij})_{ij}):=(\psi(a_{ij}))_{ij},\]
such that $\Psi\circ\phi$ is an $L$-algebra monomorphism from $\Matr_n(L)$ into $\Matr_n(\bar{L}).$  
\end{lemma}

\begin{proof}
The map from $\tens{\Matr_n(L')}{L'}{\bar{L}}$ to $\Matr_n(\bar{L})$ defined by
\[\tens{x}{L'}{y}\mapsto \phi(x)y\]
is surjective, and thus $\im(\phi)$ contains a $\bar{L}$-basis of $\Matr_n(\bar{L}).$ In particular $\phi(L)$ is a subset of the
center $\bar{L}$ of $\Matr_n(\bar{L}).$ Now choose $\psi\in\Gal(\bar{L}|L')$ such that $\psi^{-1}|_L$ equals $\phi.$ 
\end{proof}

\begin{proof}[\ref{propExEnlargedBTBClassGrps}]
A semisimple group is perfect and has therefore no characters. By remark \ref{remConnectedCompG} the only cases left are:
\be
\item $\sigma$ is unitary.
\item $\sigma$ is orthogonal and $md=2$ and not (\ref{eqCondO2is}).
\item Situation (\ref{eqCondO2is}).
\ee

Case 1: We have $D=k$ by Remark \ref{remd12}. There is an isomorphism of $\bar{k}$-algebras with involution 
\[(\tens{\End_k(V)}{k_0}{\bar{k}},\tens{\sigma}{k_0}{\id_{\bar{k}}})\cong (\Matr_m(\bar{k})\times\Matr_m(\bar{k}),\tilde{\sigma})\]
with 
\[\tilde{\sigma}(B,C)=(C^T,B^T),\] 
and we have $\bar{k}$-isomorphisms
\[\bG=\U(\tens{\sigma}{k_0}{\id_{\bar{k}}})\cong\U(\tilde{\sigma})\cong\GL_m(\bar{k}).\]
The last isomorphism is induced by the projection to the first 
coordinate. 
Let $\chi$ be a $k_0$-rational character of $\bG$.
A character of $\GL_m(\bar{k})$ is a power of the determinant. Thus, because of Lemma \ref{lemMatrixEmbedding}, $\chi|_{G}$ must be a power of $\det|_{G}$ or 
$\rho\circ\det|_{G}.$
We now fix a basis of $V$ to get a $k$-isomorphism from $\Matr_m(k)$ to $\End_k(V).$
The involution $\sigma\circ ()^{\rho}$ is conjugate to the transpose map, which implies 
\[\chi(x)^{-1}=\chi(\sigma(x))=\rho(\chi(x))=\chi(x),\] for all 
$x\in G.$ The last equality follows from $\chi(G)\subseteq k_0$. In particular the only possible values of $\chi$ on $G$ are 1 and $-1.$
Thus $\chi$ is trivial, because $\bG$ is connected and $G$ is Zariski-dense in $\bG$ by \cite[18.3]{borel:91}. 

Case 2: We have $k=k_0,$ since $\sigma$ is orthogonal. If $d=2$ There is an element $a\in\SU(h)\setminus\{1,-1\},$ because $\SU(h)$ is Zariski-dense in $\bSU(h)$ by \cite[18.3]{borel:91}. We have that $k[a]$ is its own centralizer in $D,$ because the index of $D$ is two, in particular the commutative group $\SU(h)$ is a subset of $k[a]$. if we introduce a $k$-basis of $D$ which contains 
$1$ and $a$, then the identity from  $\SU(h)$ to $\U(\sigma|_{k[a]})$ can be extended to a $k$-isomorphism from $\bSU(h)$ to $\U(\sigma|_{k[a]}).$ By case 1 there is no $k$-rational character on $\bSU(h).$

Let us now assume $d=1$ and $\bSU(h)$ is anisotropic. There is a $k$-basis of $V$ such that the Gram matrix of 
$h$ is of the form 
\[\zzmatrix{e}{0}{0}{f},\]
and we identify $\End_k(V)$ with $\Matr_2(k).$ A short calculation shows that 
\[\bSU(h)=\left\{\zzmatrix{a}{cf}{-ce}{a}\mid a,c\in\bar{k}\ s.t.\ a^2+efc^2=1\right\}.\]
We fix square roots $\sqrt{e}$ and $\sqrt{-f}.$
The conjugation with
\[\zzmatrix{\sqrt{e}}{\sqrt{-f}}{\frac{\sqrt{e}}{2}}{-\frac{\sqrt{-f}}{2}}\]
maps $\bSU(h)$ to $\bSO_2^{is}$. The explicit formula for the map is 
\[\zzmatrix{a}{cf}{-ce}{a}\mapsto \zzmatrix{a-c\sqrt{-ef}}{0}{0}{a+c\sqrt{-ef}},\ \sqrt{-ef}:=\sqrt{-e}\sqrt{f}.\]
Thus a character of $\bSU(h)$ is of the form 
\[\zzmatrix{a}{cf}{-ce}{a}\mapsto (a+c\sqrt{-ef})^z,\]
for some integer $z.$ The inverse of $(a+c\sqrt{-ef})$ is $(a-c\sqrt{-ef}).$ If $z$ is positive, we apply the binomial expansion to get coeffitients $\alpha$ und $\gamma$ in $k$ such that
\[(a+c\sqrt{-ef})^z=\alpha+\gamma\sqrt{-ef}.\] The element $\gamma$ is zero, because $\sqrt{-ef}\notin k$ since $h$ is anisotropic. Thus a $k$-rational character $\chi$ of $\bSU(h)$ satisfies 
\[\chi(x)=\chi(x^{-1}),\]
for all $x\in \SU(h).$ The density of $\SU(h)$ in $\bSU(h)$ and the connectivity of $\bSU(h)$ imply that $\chi$ is trivial.

Case 3: Here $\bG$ is $k_0$-isomorphic to $\bO_2^{is}$ implying that $\bG^0$ has non-trivial $k_0$-rational characters.
\end{proof}

\bibliographystyle{alpha}
\bibliography{skodlerack_bibliography}
\end{document}